\documentclass[10pt,twocolumn,amsmath,amssymb,aps,pra,secnumarabic,
    nofootinbib,groupedaddress]{revtex4-1}

\usepackage{mathptmx,amsthm,enumitem,tikz,tikz-cd,comment}
\usetikzlibrary{calc,decorations.markings,decorations.pathreplacing,arrows}
\usepackage[colorlinks=true,linkcolor=red!70!black,citecolor=green!70!black,
    urlcolor=magenta!70!black,backref]{hyperref}
\setlist[enumerate,1]{font=\bfseries,label=\arabic*.}

\makeatletter\def\@bibdataout@init{}\def\pre@bibdata{}\makeatother

\colorlet{darkred}{red!70!black}
\colorlet{darkblue}{blue!50!black}
\colorlet{medgreen}{green!70!black!70!white}

\newtheorem{theorem}{Theorem}[section]
\newtheorem{proposition}[theorem]{Proposition}
\newtheorem{lemma}[theorem]{Lemma}

\newcommand{\Sec}[1]{Section~\ref{#1}}
\newcommand{\Thm}[1]{Theorem~\ref{#1}}

\newcommand{\ie}{\emph{i.e.~}}

\newcommand{\defeq}{\stackrel{\mathrm{def}}=}

\newcommand{\Z}{\mathbb{Z}}

\newcommand{\shP}{\mathsf{\#P}}

\newcommand{\onto}{\twoheadrightarrow}

\renewcommand{\setminus}{\smallsetminus}

\newcommand{\sslash}{\mathbin{/\mkern-5mu/}}

\newcommand{\Inn}{\operatorname{Inn}}

\newcommand{\ord}{\operatorname{ord}}
\newcommand{\dil}{\operatorname{dil}}

\newcommand{\Maps}{\operatorname{Maps}}
\newcommand{\MCG}{\operatorname{MCG}}
\newcommand{\mcg}[1]{\operatorname{MCG}_*(\Sigma_{#1})}
\newcommand{\sfc}[1]{\Sigma_{#1}}

\newcommand{\sch}{\operatorname{sch}}
\newcommand{\e}{\epsilon}
\newcommand{\hR}{\hat{R}}

\newcommand{\ov}[1]{\overline{#1}}

\begin{document}
\title{Schur-type invariants of branched $G$-covers of surfaces}

\author{Eric Samperton}
\email{eric@math.ucsb.edu}
\affiliation{University of California, Santa Barbara}

\date{\today}

\begin{abstract}
Fix a finite group $G$ and a conjugacy invariant subset $C\subseteq G$.  Let $\Sigma$ be an oriented surface, possibly with punctures.  We consider the question of when two homomorphisms $\pi_1(\Sigma) \to G$ taking punctures into $C$ are equivalent up to an orientation preserving diffeomorphism of $\Sigma$.  We provide an answer to this question in a stable range, meaning that $\Sigma$ has enough genus and enough punctures of every conjugacy type in $C$.  If $C$ generates $G$, then we can assume $\Sigma$ has genus 0 (or any other constant).  The main tool is a classifying space for (framed) $C$-branched $G$-covers, and related homology classes we call branched Schur invariants, since they take values in a torsor over a quotient of the Schur multiplier $H_2(G)$.  We conclude with a brief discussion of applications to $(2+1)$-dimensional $G$-equivariant TQFT and symmetry-enriched topological phases.
\end{abstract}

\maketitle

\section{Introduction}
\label{s:intro}
Let $\sfc{g,n}$ denote an oriented genus $g$ surface with $n$ distinct marked points $p_1,\dots,p_n \in \sfc{g},$ thought of as punctures.  Throughout, we assume $\sfc{g,n}$ has a fixed basepoint distinct from the punctures.  Recall that the pointed mapping class group $\mcg{g,n}$ consists of isotopy classes of orientation-preserving diffeomorphisms of $\sfc{g,n}$ that fix the basepoint.  In particular, representatives of elements of $\mcg{g,n}$ are allowed to permute the punctures.  Since the basepoint of $\sfc{g,n}$ is fixed, $\mcg{g,n}$ acts on $\pi_1(\sfc{g,n})$.

Fix a finite group $G$.  The action of $\mcg{g,n}$ on $\pi_1(\sfc{g,n})$  induces an action of $\mcg{g,n}$ on the finite set of $G$-representations
\[ \hR_{g,n} \defeq \{ \pi_1(\sfc{g,n}) \to G \}. \]
The goal of the present paper is to understand the orbits of this action.

We make significant progress by deploying three different $\mcg{g,n}$-invariants of $\hR_{g,n}$.  When combined, these invariants are powerful enough that we can specify orbits uniquely in a certain ``stable range."  The precise results are provided by Theorem \ref{th:main}.  For now we simply remark that as $g$ and $n$ grow, the stable range includes almost all homomorphisms in $\hR_{g,n}$ with respect to the uniform counting measure.  After first accounting for two elementary invariants, our approach is to interpret the homomorphisms in $\hR_{g,n}$ as branched $G$-covers of $\Sigma_g$, where the punctures form the branch locus.  We then use some algebraic topology to construct invariant homology classes called branched Schur invariants.

The first invariant is the image of a homomorphism: if $\phi \in \hR_{n,g}$ has image $H \leq G$, then so does $\tau \cdot \phi$ for all $\tau \in \mcg{n,g}$.  Accordingly, we only need to consider the subset
\[ R_{g,n} \defeq \{ \pi_1(\Sigma_{g,n}) \onto G \} \subseteq \hR_{g,n} \]
consisting of all \emph{surjective} homomorphisms.  In terms of branched covers, this will mean we only consider \emph{connected} covers of $\sfc{g,n}$.  This reduces us to the narrower question: what are the orbits of the action of $\mcg{g,n}$ on $R_{g,n}$?

The second invariant is the branch type of a homomorphism, defined as follows.  We begin with some notation: given a conjugacy invariant subset $C \subseteq G$, we let $C\sslash G$ denote the set of conjugacy classes intersecting $C$.  If $c \in G$, we denote the conjugacy class of $c$ by $\ov{c} \in G\sslash G$.

For each puncture $p_i$, pick a simple closed loop $\gamma_i \in \pi_1(\sfc{g,n})$ such that $\gamma_i$ winds once counterclockwise around $p_i$, and $\gamma_i$ does not wind around any of the other punctures.  It is not quite correct to say that $\mcg{g,n}$ preserves the conjugacy class of $f(\gamma_i)$, since $\mcg{g,n}$ can permute the punctures.  However, if we form a vector
\[ v_\phi \in \Z_{\geq 0}^{G\sslash G} \]
by letting the component of $\ov{c} \in G\sslash G$ in $v_\phi$ be
\[ (v_\phi)(\ov{c}) \defeq |\{ 1 \leq i \leq n \mid \phi(\gamma_i) \in \ov{c} \}|, \]
then elementary algebraic topology shows $\mcg{g,n}$ preserves $v_\phi$.  We call $v_\phi$ the \emph{branch type}, or \emph{branching data}, of $\phi$.

We can interpret the branch type $v_\phi$ as a multiset of cardinality $n$,  meaning the sum of the entries of the vector $v_\phi$ is $n$.  Given any branching data $v \in \Z_{\geq 0}^{G\sslash G}$ of cardinality $n$, we define the $\mcg{g,n}$-invariant subset
\[ R_{g,v} \defeq \{ \phi \in R_{g,n} \mid v_\phi = v \} \subseteq R_{g,n}. \]
In fact, we will refine our approach by fixing a conjugacy invariant subset $C \subseteq G$ and only considering
\[ v \in \Z_{\geq 0}^{C\sslash G} \subseteq \Z_{\geq 0}^{G\sslash G}. \]
We emphasize that $C$ need not be closed under inversion.  In everything that follows, we allow $C = \emptyset$, although then we must set $n=0$.

In Section \ref{ss:defn}, we define the notion of \emph{(framed) $C$-branched $G$-cover} of a smooth manifold.  Then, given $\phi \in R_{g,v}$, we construct a $C$-branched $G$-cover of $\Sigma_g$ with branch locus consisting of the punctures $p_1,\dots, p_n$.  The details of this are provided in Section \ref{ss:book}.  By a small abuse of notation, we continue to denote this cover by $\phi$.

In Section \ref{ss:classifying}, we describe a classifying space for (concordance classes of) $C$-branched $G$-covers of smooth manifolds, denoted $BG_C$.  The $C$-branched $G$-cover $\phi$ yields a (homotopy class) of a map $\phi_\#: \Sigma_g \to BG_C$.  The \emph{$C$-branched Schur invariant} of $\phi$ is the integral homology class
\[ \sch_C(\phi) \defeq \phi_*[\Sigma_g] \in H_2(BG_C) \]
where $[\Sigma_g] \in H_2(\Sigma)$ is the orientation of $\Sigma_g$.  (All of the homology groups in this paper have integral coefficients.)  Since we only consider orientation-preserving mapping classes, $\sch_C(\phi)$ is $\mcg{g,n}$-invariant.

Our main theorem shows that the branched Schur invariant completely determines the orbits of $\mcg{g,n}$ acting on $R_{g,v}$ whenever the genus $g$ and branching data $v$ are ``large enough."  When we say $g$ is \emph{large enough}, we mean this is in the usual sense for integers.  However, when we say $v$ is \emph{large enough}, we mean that every conjugacy class $\ov{c} \in C\sslash G$ occurs in $v$ with enough multiplicity, \ie  that all of the integers $v(\ov{c})$ are large enough.  In particular, when we say $v$ is large enough, we do not simply mean that the cardinality of $v$ is large enough.

Before stating the main theorem, we introduce one more definition.  Recall that the group homology $H_*(G)$ is equivalent to the singular homology $H_*(BG)$, where $BG$ is the classifying space for $G$.

We say that a homology class in $H_2(G)$ is a \emph{$C$-torus} if it can be represented by a pointed map from the torus $(S^1 \times S^1,*)$ to $(BG,*)$ such that the induced map $\pi_1(S^1 \times S^1) \to G$ sends the loop winding once counterclockwise around the first factor of $S^1$ to an element of $C$.  Define the \emph{$C$-reduced Schur multiplier} of $G$ as the quotient
\[ M(G)_C \defeq H_2(G)/\langle C\text{-tori}\rangle. \]
In Section \ref{ss:homology}, we compute the low-dimensional homology of $BG_C$, and learn that $M(G)_C$ plays an important role.

\begin{theorem}
Fix a finite group $G$ and a conjugacy invariant subset $C \subseteq G$.  Let $v \in \Z_{\geq 0}^{C\sslash G}$ be a branch type of cardinality $n$.
\begin{enumerate}
\item If $v$ and $g$ are large enough, then the $C$-branched Schur invariant is a complete invariant for the orbits of the action of $\mcg{g,n}$ on $R_{g,v}$.
\item If $v$ is large enough and $C$ generates $G$, then the $C$-branched Schur invariant is a complete invariant for the orbits of the action of $\mcg{g,n}$ on $R_{g,v}$ \emph{for all} $g \geq 0$.
\end{enumerate}
In both cases, the set of orbits
\[ R_{g,v} / \mcg{g,n} \]
is a torsor for $M(G)_C$.
\label{th:main}
\end{theorem}

In principle, \Thm{th:main} gives a practical way to compute the orbits of $R_{g,v}$ in the \emph{stable range}, meaning when $v$ and $g$ are large enough.  One caveat, however, is that we do not have any upper bounds on when the stable range begins.  In other words, how large $v$ and $g$ must be in order to guarantee the conclusions of Theorem \ref{th:main} is unclear to us.  One thing is clear: the answer depends intimately on $G$ and $C$.  See Section \ref{ss:unstable} for more discussion.

Theorem \ref{th:main} is not entirely new.  On one hand, Dunfield and Thurston solved the unbranched case where $C=\emptyset$ and $n=0$ \cite{DunfieldThurston:random}.  In this case, $BG_C=BG$ and $M(G)_C = M(G) = H_2(G)$, and we recover their results.  On the other hand, Ellenberg, Venkatesh and Westerland solved the braid case  where $g=0$ and $\langle C\rangle =G$ \cite{EVW:hurwitz2}.  Their techniques build on work of Fried and Volklein \cite{FriedVolklein:galois}, who used unpublished ideas of Conway and Parker to solve the braid case with the additional assumptions that $C=G$ and $M(G)_C = 0$.  Thus, our Theorem \ref{th:main} can be understood as an interpolation between the two extremal cases of \cite{DunfieldThurston:random} and \cite{EVW:hurwitz2}.  We note that our approach via the classifying space $BG_C$ is briefly mentioned, but left undeveloped, in \cite{EVW:hurwitz2}.

\Sec{ss:main} contains the proof of \Thm{th:main}.  The first statement of the theorem follows by applying Proposition \ref{p:stable} in a manner similar to how Dunfield and Thurston  applied Livingston's stable equivalence theorem \cite{Livingston:stabilizing}.  The second statement follows from the Hopf-Whitney classification, a kind of generalization of the Hurewicz isomorphism.  The final part of the theorem follows, in part, from a computation of the second homology of $BG_C$ that we carry out in Section \ref{ss:homology}.  Our proof exploits the well-known fact that homology and oriented bordism are the same in dimension 2.  Section \ref{s:outlook} contains more remarks on the stable range, and a brief discussion of potential applications to symmetry-enriched topological phases. 

\acknowledgments{I want to thank Greg Kuperberg for helpful feedback concerning earlier drafts, in particular for bringing my attention to a gap in the first version's Lemma 3.2, as well as helping to streamline the proof of the second part of Theorem \ref{th:main} by using the Hopf-Whitney classification.  I also want to thank Greg for encouraging me to learn TikZ; if nothing else, the figures in this version should make it look better than the first version.  Finally, my thanks to the anonymous referee for patient and helpful comments.}

\section{$C$-branched $G$-covers}
\label{s:brand}
Throughout this section, let $G$ be a discrete group, and let $C \subseteq G$ be a conjugacy invariant subset, possibly empty and not necessarily inverse-closed.

\subsection{Definitions and examples}
\label{ss:defn}
Let $M$ be a smooth, connected manifold, possibly with boundary.  A \emph{$C$-branched $G$-cover of $M$} consists of the following data:
\begin{enumerate}
\item A smooth map $f:\tilde{M} \to M$, where $\tilde{M}$ is a smooth manifold.
\item A codimension 2 properly embedded submanifold $K \subseteq M$, possibly empty, called the \emph{branch locus}.
\item A framing of $K$, \ie a trivialization of the unit disk bundle
\[N(K) \cong K \times D^2, \]
where $D^2$ has the standard orientation it inherits as a submanifold of $\mathbb{R}^2$.  We will conflate $N(K)$ with a closed regular neighborhood of $K$.
\end{enumerate}
This data must satisfy the following conditions:
\begin{enumerate}
\item $f\mid f^{-1}(M\setminus K)$ is a regular $G$-cover.
\item ($C$-branched condition) The monodromy homomorphism $\pi_1(M \setminus K) \to G$ associated to the regular $G$-cover $f\mid f^{-1}(M\setminus K)$ sends a counterclockwise loop around the boundary of each fiber $D^2$ of $N(K)$ into $C$.  (``Counterclockwise" is determined by the framing of $K$.)
\item $f\mid f^{-1}(K)$ is a cover over each component of $K$.
\end{enumerate}
If $M$ is not connected, then a $C$-branched $G$-cover of $M$ is simply a $C$-branched $G$-cover of each component of $M$.  We will variously abuse notation by referring to $\tilde{M}$ or $M$ as a $C$-branched $G$-cover of $M$, and taking the other structures for granted.  Also, $C$ and $G$ will be fixed throughout, so if we say ``branched cover," we always mean $C$-branched $G$-cover.

We make several remarks regarding this definition.  First, if either $C = \emptyset$ or $K = \emptyset$, then a $C$-branched $G$-cover of $M$ is just a regular, unbranched $G$-cover of $M$.  Second, a branched cover of $M$ is connected if and only if the regular cover over $M\setminus K$ is connected if and only if the monodromy homomorphism $\pi_1(M \setminus K) \to G$ is surjective.

Finally, note that if we pick a component of $K$, then for any two fibers of $N(K)$ over that component, their counterclockwise boundary loops map to conjugate elements of $G$.  Thus, to every component of $K$ we associate a conjugacy class in $C$, called the \emph{branch type} of the component.  Relatedly, if $K \subseteq M$ does not admit a framing, then, by our definition, $M$ does not have any $C$-branched $G$-covers with branch locus $K$.  Thus, it is arguably more precise to call our covers ``\emph{framed} $C$-branched $G$-covers."  There are alternative reasonable definitions for what it means to be $C$-branched in the absence of a framing. For example, one may suppose the weaker condition that $N(K)$ admit a consistent orientation of its fibers.  However,
the classifying space $BG_C$ we exploit later classifies framed covers, so we will absorb the adjective ``framed" into the adjective ``$C$-branched."  See the end of Section \ref{ss:classifying} for more discussion.

There are three equivalence relations on $C$-branched $G$-covers that we will need: equivalence, concordance, and cobordism.  Each of these is coarser than the preceding one.  As we shall see, \Thm{th:main} can be interpreted as saying that cobordism is sometimes enough to guarantee equivalence anyway.

Two $C$-branched $G$-covers $\tilde{M}_0$ and $\tilde{M}_1$ of $M_0$ and $M_1$, respectively, are \emph{equivalent} if there is a diffeomorphism from $\tilde{M}_0$ to $\tilde{M}_1$ that takes $f_0^{-1}(K_0)$ diffeomorphically to $f_1^{-1}(K_1)$ so that the framings are identified, and that is an equivalence of $G$-covers on the complement of the branch loci.  If $M_0$ and $M_1$ are oriented, we require equivalences to preserve orientations.

Two $C$-branched $G$-covers $\tilde{M}_0$ and $\tilde{M}_1$ are \emph{cobordant} if there is a manifold $W$ such that $\partial W = M_0 \sqcup M_1$, and a $C$-branched $G$-cover $\tilde{W}$ that is equivalent to $\tilde{M}_0$ when restricted to $M_0$ and equivalent to $\tilde{M}_1$ when restricted to $M_1$.  In particular, the framing of the branch locus of $W$ has to extend the framings of $K_0$ and $K_1$.  We can also talk about oriented cobordism.  In the sequel, when we say cobordism we will always mean oriented cobordism.

If $W = M \times I$ is a cobordism, then we say $\tilde{M}_0 = M \times \{0\}$ and $\tilde{M}_1 = M \times \{1\}$ are \emph{concordant}.  Every equivalence yields a concordance by taking the mapping cylinder. Not all concordances are cylinders because there can be births and deaths of components of the branch loci.

The next lemma shows that a $C$-branched $G$-cover of $M$ is uniquely specified by $K$, a framing of $K$, and a homomorphism $\pi_1(M \setminus K) \to G$ satisfying the $C$-branched condition.  The proof explains the requirement that $N(K)$ be trivializable (or, at least, that the fibers of $N(K)$ be coherently orientable).

\begin{lemma}
Let $K$ be a codimension 2 properly embedded submanifold of $M$ such that $N(K)$ is trivializable.  Then for every choice of framing of $K$ and for every regular $G$-cover of $M \setminus K$ that satisfies the $C$-branched condition with respect to the chosen framing, there is a unique (up to equivalence) $C$-branched $G$-cover of $M$ with branch locus $K$ with the given framing and $G$-cover of $M \setminus K$.
\label{l:extension}
\end{lemma}

\begin{proof}
Let $\overline{c} \in C\sslash G$.  There is a standard way to construct a branched $G$-cover of the oriented disk $D^2$, with a single branch point of type $\overline{c}$.  Each component is modeled on the the $p^{\text{th}}$ power map $z \mapsto z^p$ on the unit disk in the complex plane, where $p$ is the order of the group element $c$; there are $|G|/p$ components.  Call this branched cover $D_{\overline{c}}^2$.

Given a trivialization of $N(K)$, let $K_{\overline{c}}$ be a component of $K$ with branch type $\overline{c}$.  Then $K_{\overline{c}} \times D_{\overline{c}}^2$ is a $C$-branched $G$-cover of $N(K_{\overline{c}})$ with branch locus $K_{\overline{c}}$.  We form a $C$-branched $G$-cover of $N(K)$ by taking the union over components of $K$.  The $C$-branched condition implies the restriction of this branched cover to $\partial N(K)$ is an unbranched regular $G$-cover that is equivalent to the cover of $\partial N(K)$ induced by the given regular $G$-cover of $M\setminus K$.  So we can glue the regular $G$-cover of $M \setminus N(K)^\circ$ to the $C$-branched $G$-cover of $N(K)$ along $\partial N(K)$.  Note that, before gluing, it may be necessary to homotope one of the covers in a neighborhood of $\partial N(K)$ to guarantee the result is smooth.  This yields a $C$-branched $G$-cover of $M$ with branch locus $K$ with the given trivialization of $K$ and $G$-cover of $M \setminus K$.  It is straightforward to verify that any other branched cover of $M$ with this data is equivalent.
\end{proof}

\begin{figure*}
\begin{tikzpicture}[decoration={markings,
    mark=at position 0.4 with {\arrow{angle 90}}},scale=0.8]
\draw[very thick] (3,0) circle [radius = 3];
\draw[very thick] (0,0) arc [x radius = 3, y radius = 1.5, start angle = 180, end angle = 360];
\draw[very thick,dashed] (0,0) arc [x radius = 3, y radius = 1.5, start angle = 180, end angle = 0];
\fill[medgreen] (3,2.6) circle [radius=.08];
\draw[below left] (3,2.6) node {$+$};
\fill[medgreen] (3,-2.6) circle [radius=.08];
\draw[above left] (3,-2.6) node {$-$};
\draw[very thick, white,double=black] (0,0) .. controls (3.5,2.5) .. (3.4,2.8) .. controls (3.3,3.1) and (2.5,3) .. (0,0);
\draw[thick,postaction={decorate}] (0,0) .. controls (3.5,2.5) .. (3.4,2.8) .. controls (3.3,3.1) and (2.5,3) .. (0,0);
\draw[thick,postaction={decorate}] (0,0) .. controls (3.5,-2.5) .. (3.4,-2.8) .. controls (3.3,-3.1) and (2.5,-3) .. (0,0);
\fill (0,0) circle [radius=.1];
\draw (3,1.8) node {\large $c$};
\draw (3.1,-1.8) node {\large $c$};
\begin{scope}[xshift = 240]
\draw[very thick] (3,0) circle [radius = 3];
\draw[very thick] (0,0) arc [x radius = 3, y radius = 1.5, start angle = 180, end angle = 360];
\draw[very thick,dashed] (0,0) arc [x radius = 3, y radius = 1.5, start angle = 180, end angle = 0];
\fill[medgreen] (3,2.6) circle [radius=.08];
\draw[below left] (3,2.6) node {$+$};
\fill[medgreen] (3,-2.6) circle [radius=.08];
\draw[above left] (3,-2.6) node {$+$};
\draw[very thick, white,double=black] (0,0) .. controls (3.5,2.5) .. (3.4,2.8) .. controls (3.3,3.1) and (2.5,3) .. (0,0);
\draw[thick,postaction={decorate}] (0,0) .. controls (3.5,2.5) .. (3.4,2.8) .. controls (3.3,3.1) and (2.5,3) .. (0,0);
\draw[thick,postaction={decorate}] (0,0) .. controls (2.5,-3) and (3.3,-3.1) .. (3.4,-2.8) .. controls (3.5,-2.5) .. (0,0);
\fill (0,0) circle [radius=.1];
\draw (3,1.8) node {\large $c$};
\draw (3.3,-1.8) node {\large $c^{-1}$};
\end{scope}
\end{tikzpicture}
\caption{Two examples of $C$-branched $G$-covers of $S^2$.  The branch locus $K$ of both covers (indicated in green) consists of two points.  The framing of $K$ is indicated by signs on the branch points.  The left cover can be extended into the 3-ball by adding an oriented arc that connects the two branch points and using the blackboard framing.  Thus, the left cover is null-cobordant.  The right cover is not null-cobordant, because the two branch points have the same framing.  Note that this is true even if $c=c^{-1}$, and even though the unbranched $G$-covers of $S^2 \setminus K$ are the same.  Thus we see the choice of framing can affect the equivalence class of a $C$-branched $G$-cover whenever $C$ is nonempty.}
\label{f:trivializations}
\end{figure*}
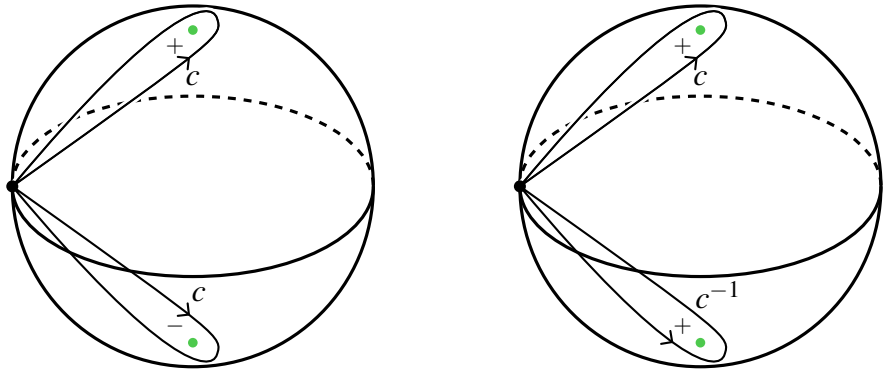

\begin{figure}
\begin{tikzpicture}[xscale=1, yscale=.9,decoration={markings,
    mark=at position 0.75 with {\arrow{angle 90}}}]
\fill[blue!10] (5,0) arc[x radius = 1.035, y radius = .6, start angle=0, end angle=-360];
\draw[thick,blue,dashed] (5,0) arc[x radius = 1.035, y radius = .6, start angle=0, end angle=-180];
\draw[medgreen,thick,postaction={decorate}] (4,0) arc [x radius=2, y radius=.85,start angle=360, end angle=0];
\fill[medgreen] (4,0) circle [radius = 0.04];
\draw[blue,postaction={decorate},thick] (5,0) arc[x radius = 1.035, y radius = .6, start angle=0, end angle=180];
\draw[thick,red,postaction={decorate}] (2,0) circle [x radius=2.85, y radius=1.35];
\fill (4.7,.43) circle [radius = .04];
\draw[very thick] (2,0) circle [x radius=3, y radius=1.5];
\draw[very thick] (1.07,-0.05)  arc[x radius = .95, y radius = .5, start angle=170, end angle=11];
\draw[very thick] (1,.1)  arc[x radius = 1, y radius = .5, start angle=185, end angle=355];
\end{tikzpicture}
\caption{This figure shows that every $C$-torus is null-cobordant as a $C$-branched $G$-cover. The red and blue curves represent the standard generators of $\pi_1(S^1 \times S^1)$.  Suppose the blue curve is mapped to $a \in C$ and the red curve is mapped to some element $b \in G$ that commutes with $a$.  We can extend this (unbranched) $G$-cover of $S^1 \times S^1$ to a $C$-branched $G$-cover of the solid torus $D^2 \times S^1$.  The branch locus $K$ is the core, indicated by the green curve, equipped with the blackboard framing.}
\label{f:torus}
\end{figure}
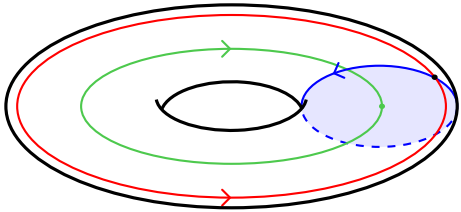

We conclude this subsection with some examples of $C$-branched $G$-covers.  We first explain how we shall describe the covers.  Lemma \ref{l:extension} says we can specify a well-defined $C$-branched $G$-cover of $M$ by specifying $K$, a trivialization of $N(K)$, and a $G$-cover of $M \setminus K$ satisfying the $C$-branched condition.  We describe the $G$-cover of $M \setminus K$ by picking a basepoint and fixing a homomorphism $\pi_1(M \setminus K) \to G$.

When $M$ is 2-dimensional, $K$ is simply a collection of points, which we will call \emph{branch points}.  The disk bundle $N(K)$ is the union of the unit tangent disks at $K$.  Up to oriented equivalence, a trivialization of a 2-disk bundle over a point is just an orientation of the disk.  We conclude that if $M$ is oriented, we can specify a framing of $K$ by labelling each branch point in $K$ with a sign $+$ or $-$ to indicate whether the trivialization of the normal disk agrees with the orientation of $M$ or not.

Similar remarks apply when $M$ is oriented and 3-dimensional.  In this case, $K$ is a link in $M$, and we can specify a framing of $K$ by assigning an integer to each component.  In practice, when drawing figures, we will always use the blackboard framing, which is determined by an orientation.

Figure \ref{f:trivializations} provides two different $C$-branched $G$-covers of $S^2$ with the same unbranched $G$-covers over $S^2 \setminus K$.  These examples show how a $C$-branched $G$-cover depends on the choice of framing.

Figure \ref{f:torus} shows that every $C$-torus is null-cobordant when considered as a $C$-branched $G$-cover of $S^1 \times S^1$.  In fact, we can generalize this example.

\begin{lemma}
If $C$ generates $G$, then every unbranched $G$-cover of a closed, oriented surface is cobordant to a branched cover of $S^2$.
\label{l:unbranched}
\end{lemma}

\begin{proof}
We begin by showing that every unbranched $G$-cover of a torus with boundary is cobordant to a $C$-branched $G$-cover of the disk $D^2$.  Consider Figure \ref{f:unbranched}.  Suppose the blue $\alpha$ curve is mapped to $a \in G$ and the red $\beta$ curve is mapped to $b \in G$.  This describes an unbranched $G$-cover of the torus with boundary.  We glue on a $C$-branched $G$-cover of a 2-handle as indicated in the figure.  The branch locus $K$ is given by the green curves, and the framing of $K$ is understood to be the blackboard framing determined by the orientations of the curves.  Send the loops $\gamma_i$ to any $x_i \in C$ so that
\[ x_k\cdots x_1=a, \]
and send the $\delta_i$ to $y_i$ in $C$ so that
\[ y_l \cdots y_1=b. \]
This construction can now be used to construct the desired cobordism.

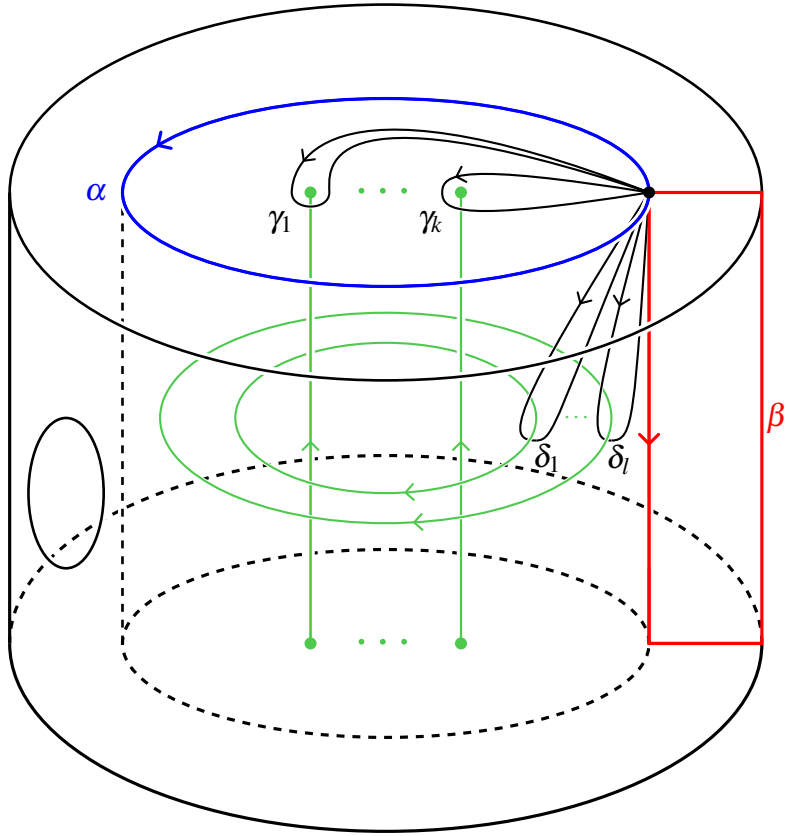
\begin{figure*}
\begin{tikzpicture}[decoration={markings,mark=at position 0.45 with {\arrow{angle 90}}}]
\foreach \i in {1,2,3}
	{
\coordinate (k\i) at ($(-2,0)+\i*(1,0)$);
}
\coordinate (p) at (3.5,6);
\draw[very thick] (5,0) arc [x radius = 5, y radius = 2.5, start angle = 360, end angle = 180];
\draw[very thick,dashed] (5,0) arc [x radius = 5, y radius = 2.5, start angle = 0, end angle = 180];
\draw[very thick, white,double=black,double distance=1,dashed] (-3.5,0) -- (-3.5,6);
\draw[very thick,dashed] (0,0) circle [x radius = 3.5, y radius=1.25];
\draw[very thick] (0,6) circle [x radius = 3.5, y radius=1.25];

\foreach \i in {1,2,3}
	{
\coordinate (c\i) at ($(1.5,3)+\i*(0.5,0)$);
}

\draw[medgreen,thick] (c1) arc [x radius = 2, y radius = 1, start angle = 0, end angle=180];
\draw[medgreen,thick] (c3) arc [x radius = 3, y radius = 1.4, start angle = 0, end angle=180];

\foreach \i in {1,3}
	{
\draw[very thick, white,double=green,double distance=0.8] (k\i) -- ($(k\i)+(0,6)$);
\draw[thick,medgreen,postaction={decorate}] (k\i) -- ($(k\i)+(0,6)$);
}

\draw[ultra thick,white] (p) .. controls ($(c1) - (.4,0)$) and ($(c1) - (.4,.3)$) .. ($(c1) - (0,.3)$);
\draw[ultra thick,white] (p) .. controls ($(c3) - (.3,0)$) and ($(c3) - (.3,.3)$) .. ($(c3) - (0,.3)$);
\draw[thick,postaction={decorate}] (p) .. controls ($(c1) - (.4,0)$) and ($(c1) - (.4,.3)$) .. ($(c1) - (0,.3)$);
\draw[thick,postaction={decorate}] (p) .. controls ($(c3) - (.3,0)$) and ($(c3) - (.3,.3)$) .. ($(c3) - (0,.3)$);

\draw[white, ultra thick] (c1) arc [x radius = 2, y radius = 1, start angle = 0, end angle=-180];
\draw[white, ultra thick] (c3) arc [x radius = 3, y radius = 1.4, start angle = 0, end angle=-180];
\draw[medgreen,thick,postaction={decorate}] (c1) arc [x radius = 2, y radius = 1, start angle = 0, end angle=-180];
\draw[medgreen,thick,postaction={decorate}] (c3) arc [x radius = 3, y radius = 1.4, start angle = 0, end angle=-180];

\draw[very thick, white,double=red,double distance=1] (p) -- (3.5,1);
\draw[ultra thick,white] (p) .. controls ($(c1) + (.25,0)$) and ($(c1) + (.3,-.3)$) .. ($(c1) - (0,.3)$);
\draw[ultra thick,white] (p) .. controls ($(c3) + (.3,0)$) and ($(c3) + (.3,-.3)$) .. ($(c3) - (0,.3)$);
\draw[thick] (p) .. controls ($(c1) + (.25,0)$) and ($(c1) + (.25,-.3)$) .. ($(c1) - (0,.3)$);
\draw[thick] (p) .. controls ($(c3) + (.3,0)$) and ($(c3) + (.3,-.3)$) .. ($(c3) - (0,.3)$);

\draw[very thick, white,double=black,double distance=1] (-4.25,2) circle [x radius = 0.5, y radius=1];
\draw[very thick, white,double=black,double distance=1] (0,6) circle [x radius = 5, y radius=2.5];
\draw[very thick,red,postaction={decorate}] (p) -- (3.5,0) -- (5,0);
\draw[very thick, white,double=blue,double distance=1] (p) arc [x radius=3.5,y radius=1.25, start angle = 0, end angle=360];
\draw[very thick,blue,postaction={decorate}] (p) arc [x radius=3.5,y radius=1.25, start angle = 0, end angle=360];
\draw[very thick, white,double=black,double distance=1] (0,6) circle [x radius = 5, y radius=2.5];
\draw[very thick] (-5,0) -- (-5,6);
\draw[very thick,red] (3.5,1) --(3.5,0) -- (5,0) -- (+5,6) --(p);

\draw[ultra thick, white] (p) .. controls ($(k1)+(0,1.75)+(0,6)$) and ($(k1)+(-.25,.25)+(0,6)$) .. ($(k1)-(.25,0)+(0,6)$) .. controls ($(k1)+(-.25,-.25)+(0,6)$) and ($(k1)+(.25,-.25)+(0,6)$) .. ($(k1)+(.25,0)+(0,6)$) .. controls ($(k1)+(.25,.25)+(0,6)$) and ($(k1)+(0,1.5)+(0,6)$) .. (p);
\draw[ultra thick,white] (p) .. controls ($(k3)+(0,0.4)+(0,6)$) and ($(k3)+(-.25,.25)+(0,6)$) .. ($(k3)-(.25,0)+(0,6)$) .. controls ($(k3)+(-.25,-.25)+(0,6)$) and ($(k3)-(0,0.4)+(0,6)$) .. (p);
\draw[thick, postaction={decorate}] (p) .. controls ($(k1)+(0,1.75)+(0,6)$) and ($(k1)+(-.25,.25)+(0,6)$) .. ($(k1)-(.25,0)+(0,6)$) .. controls ($(k1)+(-.25,-.25)+(0,6)$) and ($(k1)+(.25,-.25)+(0,6)$) .. ($(k1)+(.25,0)+(0,6)$) .. controls ($(k1)+(.25,.25)+(0,6)$) and ($(k1)+(0,1.5)+(0,6)$) .. (p);
\draw[thick, postaction={decorate}] (p) .. controls ($(k3)+(0,0.4)+(0,6)$) and ($(k3)+(-.25,.25)+(0,6)$) .. ($(k3)-(.25,0)+(0,6)$) .. controls ($(k3)+(-.25,-.25)+(0,6)$) and ($(k3)-(0,0.4)+(0,6)$) .. (p);

\foreach \i in {1,3}
	{
\fill[medgreen] (k\i) circle[radius=0.08];
\fill[medgreen] ($(k\i)+(0,6)$) circle[radius=0.08];
}
\draw ($(k1)+(-0.4,5.6)$) node {\large $\gamma_1$};
\draw ($(k3)+(-.4,5.6)$) node {\large $\gamma_k$};
\draw ($(c1)+(.15,-.525)$) node {\large $\delta_1$};
\draw ($(c3)+(.1,-.55)$) node {\large $\delta_l$};
\draw[blue] (-3.85,6) node {\large $\alpha$};
\draw[red] (5.2,3) node {\large $\beta$};

\draw[medgreen] (k2) node {\huge $\cdots$};
\draw[medgreen] ($(c2)+(.05,0)$) node {$\cdots$};
\draw[medgreen] ($(k2)+(0,6)$) node {\huge $\cdots$};
\fill (p) circle [radius=0.08];

\end{tikzpicture}
\caption{Surgering a $G$-handle.  The black circle on the left of the figure is a boundary component.}
\label{f:unbranched}
\end{figure*}

To prove the lemma, just do the above on every handle.  Precisely, let $\alpha_1,\dots,\alpha_g$ be a maximal collection of non-isotopic, disjoint, essential simple closed curves on $\Sigma_g$, and glue a branched cover of a 2-handle along each $\alpha_i$ as in the previous paragraph.
\end{proof}

\subsection{Classifying space}
\label{ss:classifying}
We now construct a classifying space $BG_C$ such that homotopy classes of maps from $M$ to $BG_C$ are naturally bijective with concordance classes of $C$-branched $G$-covers of $M$.  The first such constructions were provided by Brand \cite{Brand:branched}, who considered the case of irregular $n$-sheeted branched covers with various restrictions on the branch types.  The covers he considers can be identified with regular branched $G$-covers where $G=S_n$, and $BG_C$ recovers Brand's constructions in this case.  The more general construction we use was first described, to our knowledge, by Ellenberg, Venkatesh and Westerland in \cite{EVW:hurwitz2}, where they use the notation $A(G,C)$ for $BG_C$.  Since \cite{EVW:hurwitz2} was never published, and they do not provide the proof that $BG_C$ is a classifying space, we do so here.

Let $LBG = \Maps(S^1,BG)$ be the loop space of $BG$, which is the space of all maps from the circle $S^1$ into $BG$ with the compact-open topology.  The set of components of $LBG$ is the space of free homotopy classes of maps $S^1 \to BG$, which after orienting $S^1$ can be naturally identified with the set of conjugacy classes $G\sslash G$.  Let $L^CBG$ be the union of components of $LBG$ corresponding to $C\sslash G$.  Then
\[ BG_C \defeq BG \bigsqcup_{ev} L^CBG \times D^2, \]
where
\[ ev: L^CBG \times S^1 \to BG \]
is the evaluation map.  Intuitively, for every loop in $BG$ representing a free homotopy class in $C\sslash G$, we glue a disk to $BG$ so that the boundary of the disk is sent to that loop.  Our next theorem formalizes what we mean when we say that $BG_C$ is a classifying space.

\begin{theorem}
Fix a discrete group $G$ and a conjugacy invariant subset $C \subseteq G$.  Let $M$ be a smooth manifold.  Then homotopy classes of maps from $M$ to $BG_C$ are in natural bijection with concordance classes of $C$-branched $G$-covers of $M$.
\label{th:brand}
\end{theorem}

\begin{proof}
Let $\phi: M \to BG_C$ be a continuous map.  Fix $0<\e<1$ and let $D_\e(0) \subseteq D^2$ be a small disk of radius $\e$ centered at $0$.   Let $\pi: L^CBG \times D^2 \to D^2$ be the projection map.  The Whitney approximation theorem (see \cite[Thm.~6.19]{Lee:smooth}) implies that after applying some homotopy to $\phi$, we can guarantee that
\[ \pi\circ\phi: (\pi\circ\phi)^{-1}(D_\e(0))  \to D_\e(0) \]
is smooth.  By Sard's theorem, $\pi\circ\phi$ has a regular value $z \in D_\e(0)$.  The preimage $\phi^{-1}(L^CBG \times \{z\})$ is a codimension two submanifold $K \subseteq M$.  Since $N(K)$ is the pullback of the normal bundle of $z$ inside $D^2$, which is trivial, we get a trivialization of $N(K)$.  Note that 
\[ BG_C \setminus L^CBG \times \{z\} \]
is homotopy equivalent to $BG$.  Thus
\[ \phi: M \setminus K \to BG_C \setminus L^CBG \times \{z\} \]
classifies some $G$-cover of $M \setminus K$.  Clearly this cover satisfies the $C$-branched condition.  Applying Lemma \ref{l:extension}, this $G$-cover of $M \setminus K$ and the trivialization of $N(K)$ assemble to form a $C$-branched $G$-cover over $M$ with branch locus $K$.

Let $\phi$ and $\psi$ be continuous maps $M \to BG_C$, and let
\[ \Phi: M \times [0,1] \to BG_C \]
be a homotopy from $\phi$ to $\psi$.  Without loss of generality, we assume that there is an $0<\e<1$ so that
\[ \pi\circ\phi: (\pi\circ\phi)^{-1}(D_\e(0))  \to D_\e(0) \]
and
\[ \pi\circ\psi: (\pi\circ\psi)^{-1}(D_\e(0))  \to D_\e(0) \]
are both smooth.  Homotope $\Phi$ rel $M \times \{0,1\}$ so that
\[ \pi\circ\Phi: (\pi\circ\Phi)^{-1}(D_\e(0))  \to D_\e(0) \]
is smooth, and construct a $C$-branched $G$-cover of $M \times [0,1]$ as in the previous paragraph.  The resulting cover is a concordance between the covers associated to $\phi$ and $\psi$.  Thus the homotopy class of $\phi$ yields a well-defined concordance class of $C$-branched $G$-covers of $M$.

Conversely, let $\tilde{M}$ be a $C$-branched $G$-cover of $M$ with branch locus $K$.  Then there is a map $\phi: M \setminus N(K)^\circ \to BG$ classifying the $G$-cover of $M\setminus N(K)^\circ$.  The trivialization of $N(K)$ is equivalent to a map $N(K) \to D^2$ that takes $K$ to 0 and is a diffeomorphism on the fibers of $N(K)$.  Define a map $N(K) \to L^CBG$ by sending a point in $N(K)$ to the loop in $BG$ determined by $\phi$ and the boundary of the fiber that point lives in (the parametrization of the loop is determined by the trivialization).  The product of these maps from $N(K)$ is a map $N(K) \to L^CBG \times D^2$ that agrees with $\phi$ on $\partial N(K)$ when composed with the evaluation map $L^CBG \times S^1 \to BG$.  Thus $\tilde{M}$ determines a map $M \to BG_C$.  Applying this construction to a concordance between two branched covers of $M$ yields a homotopy between the corresponding maps to $BG_C$.
\end{proof}

It is interesting to consider how to modify the construction of $BG_C$ so that framings are no longer induced by maps to $BG_C$.  We expect that instead of gluing a copy of $D^2$ to $BG$ for every map $S^1 \to BG$ in the conjugacy class of $c \in G$, one should glue a copy of a cone over the classifying space $B\langle c \rangle$ of the cyclic group generated by $c$ for every map $B\langle c \rangle \to BG$ in the appropriate homotopy class.  One might call this new space $BG_C^{unfr}$.  Since, \emph{a priori}, the statement of Theorem \ref{th:main} is not concerned with framings, it is reasonable to ask for a proof that does not unnecessarily introduce framings, and, hence, a more thorough development of $BG_C^{unfr}$.  However, this would require more algebraic topology than we want to invoke, and so, for expediency's sake, we will not pursue such a direction in this paper; instead, we employ a geometric technique to circumvent framings in Section \ref{ss:dilation}.

\subsection{Invariants of $BG_C$}
\label{ss:homology}
Let us compute some basic topological invariants of $BG_C$.

First we consider $L^CBG$.  Clearly
\[ |\pi_0(L^CBG)| = |C\sslash G|. \]
A point $p$ in $L^{\overline{c}}BG$ is a loop in the free homotopy class $\overline{c}$.  A loop in $L^{\overline{c}}BG$ with basepoint $p$ is a map from a torus $S^1 \times S^1$ to $BG$ which restricts to $p$ on the meridian $S^1 \times \{1\}$.  If we fix a basepoint of $BG$, then we can pick $p$ so that its \emph{pointed} homotopy class represents $c \in \overline{c}$.  A little more work shows

\begin{lemma}
Let $G$ be a discrete group and let $\overline{c} \in G\sslash G$ be a conjugacy class in $G$ with representative $c$.  Then $\pi_1(L^{\overline{c}}BG) \cong Z_G(c)$, the centralizer of $c$ in $G$.  In particular, $H_1(L^{\overline{c}}BG) \cong Z_G(c)_{ab}$.   \qed
\label{l:LBG}
\end{lemma}

In fact, $L^{\overline{c}}BG$ is a $K(Z_G(c),1)$ (see, for instance, \cite{someone}). Thus the homology of $L^{\overline{c}}BG$ is the same as the group homology of $Z_G(c)$.  We will not use this fact, rather we simply state it so that the following lemma may be interpreted as a satisfying calculation of $H_n(BG_C)$.

\begin{lemma}
Let $G$ be a discrete group and let $C \subseteq G$ be a conjugacy invariant subset.  Then the homology of $BG_C$ fits into an exact sequence
\[\begin{aligned}
\cdots \to H_{n-1}(L^CBG) &\to H_n(BG) \to \\ &H_n(BG_C) \to H_{n-2}(L^CBG) \to \cdots. \end{aligned}\]
\label{l:exact}
\end{lemma}

\begin{proof}
Decompose $BG_C$ into the two subspaces
\[ BG \bigsqcup_{ev} L^CBG \times (D^2 \setminus \{0\}) \]
and
\[ L^CBG \times D_\epsilon(0) \]
where $D_\epsilon(0) \subseteq D^2$ is a small disk of radius $0 < \epsilon < 1$ centered at $0$.  The former subspace deformation retracts onto $BG$, while the latter subspace deformation retracts to $L^CBG \times \{0\} \simeq L^CBG$.  The intersection of the two subspaces is homotopy equivalent to $L^CBG \times S^1$, with the inclusion
\[ i: L^CBG \times S^1 \to BG \]
the evaluation map, and the inclusion
\[ j: L^CBG \times S^1 \to L^CBG \]
the projection onto the first factor.

Consider the Mayer-Vietoris sequence
\[\begin{aligned}
\cdots \to H_n(L^CBG \times &S^1) \to H_n(L^CBG) \oplus H_n(BG) \to \\
&H_n(BG_C) \to H_{n-1}(L^CBG \times S^1) \to \cdots. \end{aligned}\]
The K\"{u}nneth formula shows
\[ H_n(L^CBG \times S^1) \cong H_{n-1}(L^CBG) \oplus H_n(L^CBG). \]
The induced map
\[ j_*: H_n(L^CBG \times S^1) \to H_n(L^CBG) \]
simply projects out $H_{n-1}(L^CBG)$.  Thus, by exactness, we can modify the last term of the sequence, yielding
\[\begin{aligned}
\cdots \to H_{n-1}(L^CBG) &\oplus H_n(L^CBG) \to H_n(L^CBG) \oplus H_n(BG) \to 
\\ &H_n(BG_C) \to H_{n-2}(L^CBG) \to \cdots. \end{aligned}\]
The image of $H_n(L^CBG)$ inside $H_n(L^CBG) \oplus H_n(BG)$ is the graph of the induced map $H_n(L^CBG) \to H_n(BG)$.  The quotient of a direct sum by the graph of a linear map is the codomain, so we get the new exact sequence
\[\begin{aligned}
\cdots \to H_{n-1}(L^CBG) &\to H_n(BG) \to \\ &H_n(BG_C) \to H_{n-2}(L^CBG) \to \cdots. \end{aligned}\]
\end{proof}

Let us analyze this exact sequence in low dimensions.  When $n=1$, we have
\[ \mathbb{Z}^{C\sslash G} \to H_1(BG) \to H_1(BG_C) \to 0. \]
Of course, $H_1(BG) = H_1(G) = G_{ab}$, the abelianization of $G$.  The map $\mathbb{Z}^{C\sslash G} \to G_{ab}$ is the evaluation map defined by taking an element of $C\sslash G$ to its image in $G_{ab}$.  We deduce
\[ H_1(BG_C) \cong (G/\langle C \rangle)_{ab}. \]
Note that because $C$ is conjugacy invariant, $\langle C \rangle$ is automatically normal.  In fact, it is a straightforward exercise to compute $\pi_1(BG_C)$.

\begin{lemma}\label{l:pi1}
For any discrete group $G$ and conjugacy invariant subset $C \subset G$, $\pi_1(BG_C) \cong G/\langle C \rangle$. \qed 
\end{lemma}

Consider the exact sequence of Lemma \ref{l:exact} when $n=2$, and substitute the result of Lemma \ref{l:LBG}:
\[ \bigoplus_{\overline{c} \in C\sslash G} Z_G(c)_{ab} \to H_2(G) \to H_2(BG_C) \to \mathbb{Z}^{C\sslash G}. \]
Let $N$ be the kernel of the evaluation map
\[\mathbb{Z}^{C\sslash G} \to G_{ab}.\]
Then by the $n=1$ discussion, we can extend the sequence to
\[ \bigoplus_{\overline{c} \in C\sslash G} Z_G(c)_{ab} \to H_2(G) \to H_2(BG_C) \to N \to 0. \]

To interpret this sequence, we use the fact that oriented bordism is the same as homology in dimension 2.  Combining this with Theorem \ref{th:brand}, every element of $H_2(BG_C)$ can be represented by a $C$-branched $G$-cover of a closed, oriented surface $\Sigma$.  The map $H_2(BG_C) \to N$ takes this cover of $\Sigma$ to its \emph{homological branch type}, which algebraically counts the different conjugacy types of the branch points, with a sign that depends on whether the trivialization of the normal plane at a branch point agrees with the orientation of $\Sigma$ or not.  In particular, the homological branch type can easily be computed from the usual branch type.  (See Section \ref{ss:book} for more details on branch types.)

The image of $Z_G(c)_{ab}$ inside $H_2(G)$ can be understood as those homology classes in $H_2(G)$ that are represented by sums of tori where the meridian maps to the conjugacy class $\overline{c}$, that is, $\overline{c}$-tori.  Thus, $H_2(BG_C)$ fits into the exact sequence
\[ 0 \to M(G)_C \to H_2(BG_C) \to N \to 0, \]
where, recall from the introduction, the $C$-branched Schur multiplier of $G$ is defined as
\[ M(G)_C = H_2(G)/\langle \text{$C$-tori} \rangle. \]
Since $N \subseteq \mathbb{Z}^{C\sslash G}$ is free abelian, the sequence splits (but not naturally).  If $G$ is finite, then we can identify $M(G)_C$ with the torsion subgroup of $H_2(BG_C)$:
\[ M(G)_C \cong H_2(BG_C)_{tor}. \]

\section{Stable orbits}
\label{s:stable}
In this section we prove Theorem \ref{th:main}.

\subsection{Framed branching data and branched Schur invariants}
\label{ss:book}
Let $\phi$ be a homomorphism in $R_{g,v}$ and let $K=\{p_1,\dots,p_n\} \subseteq \Sigma_g$ be the set of punctures.  To associate a $C$-branched $G$-cover of $\Sigma_g$ to $\phi$, we must first choose a framing of $K$.  As discussed at the end of Section \ref{ss:defn}, we can specify a framing simply by decorating each point $p_i$ with a sign $o_i \in \{+1,-1\}$.

The simplest choice of framing is the \emph{positive framing} with $o_i=+1$ for all $i=1,\dots,n$.  If $v \in \Z_{\geq 0}^{C\sslash G}$ has cardinality $n$, it follows immediately from the definition of $R_{g,v}$ in Section \ref{s:intro} that the homomorphism $\phi$ satisfies the $C$-branched condition with respect to the positive framing.  Thus, by Lemma \ref{l:extension}, we can form a $C$-branched $G$-cover of $\Sigma_g$ with branch locus $K$.

Theorem \ref{th:main} is only concerned with homomorphisms in $R_{g,v}$, and, hence, branched covers with positive trivializations.  However, since our proof uses $BG_C$, it requires us to study covers with trivializations that are not positive.  The definition of branching data we provided in Section \ref{s:intro} is only sufficient for positive framings, so we now define and introduce notation for more general framings.

Let
\[ \begin{aligned} T: K &\to \{+1,-1 \} \\ p_i &\mapsto o_i \end{aligned} \]
denote a framing of $K$ and pick a homomorphism
\[ \phi: \pi_1(\Sigma_{g,n}) \to G \]
that satisfies the $C$-branched condition with respect to $T$.  Lemma \ref{l:extension} shows the pair $(T,\phi)$ determines a $C$-branched $G$-cover of $\Sigma_g$ with branch locus $K$, which we will often refer to simply by $(T,\phi)$.  We define the \emph{branching data} of this cover to be the vector
\[ v_{(T,\phi)} \in \Z_{\geq 0}^{C\sslash G \times \{+1,-1\}} \]
such that
\[ (v_{(T,\phi)})(\ov{c},o) \defeq |\{ 1 \leq i \leq n \mid o_i = o, \phi(\gamma_i)^o \in \ov{c} \}|, \]
where $\ov{c}$ is a conjugacy class in $C\sslash G$ and $o \in \{+1,-1\}$.  Thus, $(v_{(T,\phi)})(\ov{c},o)$ counts the number of branch points such that winding around them in the direction $o$ yields monodromy in $\ov{c}$.

In Section \ref{s:intro} we defined the branching data of a homomorphism in $R_{g,n}$.  Implicit in that definition was the use of the positive framing of $K$.  From now on, we make the identification
\[  \Z_{\geq 0}^{C\sslash G} = \Z_{\geq 0}^{C\sslash G \times \{+1\}} \subseteq \Z_{\geq 0}^{C\sslash G \times \{+1,-1\}}. \]
So our new definition generalizes the previous definition: $v_{(T,\phi)} = v_\phi$ if $T$ is the constant function $+1$, \ie a positive framing.

As before, we say that a vector in $v$ in $\Z_{\geq 0}^{C\sslash G \times \{+1,-1\}}$ has cardinality $n$ if the sum of its entries is $n$.  Given any $v$ of cardinality $n$, we define
\[ R_{g,v} \defeq \{ (T,\phi) \mid \phi \text{ onto and $C$-branched w.r.t. $T$, } v_{(\phi,T)} = v \}. \]
The pointed mapping class group $\MCG_*(\Sigma_{g,n})$ acts on $R_{g,v}$.  The action on the $\phi$ component is the usual action of $\MCG_*(\Sigma_{g,n})$ on a homomorphism, and the action on the $T$ component is induced by the permutation action on $K$.  If
\[ v \in \Z_{\geq 0}^{C\sslash G} = \Z_{\geq 0}^{C\sslash G \times \{+1\}} \subseteq \Z_{\geq 0}^{C\sslash G \times \{+1,-1\}} \]
then $T$ must be the positive trivialization, and the definition of $R_{g,v}$ given is the same as the definition of $R_{g,v}$ given in Section \ref{s:intro}.

It is useful to interpret $R_{g,v}$ as a parametrization of the set of equivalence classes of connected $C$-branched $G$-covers of $\Sigma_g$ with branching data $v$.  Indeed, every connected $C$-branched $G$-cover of $\Sigma_g$ is equivalent to a cover specfied by an element of $R_{g,v}$ for some unique $v$.  Of course, there could be many such elements of $R_{g,v}$; the goal of this paper is to understand them.

By Theorem \ref{th:brand}, the cover $(T,\phi) \in R_{g,v}$ induces a homotopy class of a map
\[ (T,\phi)_\#: \Sigma_g \to BG_C \]
which in turn induces a homomorphism
\[ (T,\phi)_*: H_2(\Sigma_g) \to H_2(BG_C). \]
The \emph{$C$-branched Schur invariant of $(T,\phi)$} is the homology class
\[ \sch_C(T,\phi) \defeq (T,\phi)_*\left([\Sigma_g]\right) \in H_2(BG_C), \]
where $[\Sigma_g]$ is the orientation of $\Sigma_g$.

\subsection{Stable equivalence lemmas}
\label{ss:stable}
We now explain the role of the Schur invariant.

Let $(T,\phi)$ be a $C$-branched $G$-cover of $\Sigma_g$.  A \emph{handle stabilization} of $(T,\phi)$ is any $C$-branched $G$-cover of a surface that is equivalent to the connect sum of $(T,\phi)$ with the trivial $C$-branched $G$-cover over the torus $S^1 \times S^1$.

For any $\overline{c} \in C\sslash G$, let $S^2_{\overline{c}}$ denote any $C$-branched $G$-cover of the oriented sphere $S^2$ such that the branch locus consists of one point with branch type $(\overline{c},+1)$ and one point with branch type $(\overline{c},-1)$.  For example, see the cover on the left side of Figure \ref{f:trivializations}.  If the conjugacy class $\overline{c}$ has more than one element, then there are inequivalent covers which we denote $S^2_{\overline{c}}$.  However, there will never be any ambiguity because of how we use these covers, which we now explain.

A \emph{$\overline{c}$-stabilization} of $(T,\phi)$ is any $C$-branched $G$-cover of a surface that is equivalent to the connect sum of $(T,\phi)$ with a cover of the form $S^2_{\overline{c}}$.  Note that if $(T,\phi)$ is connected (\ie $\phi$ is surjective), it does not matter which $S^2_{\overline{c}}$ we use, since stabilizing by any of them yields equivalent covers.

A \emph{puncture stabilization} of $(T,\phi)$ is some sequence of $\overline{c}$-stabilizations of $(T,\phi)$ with various $\overline{c} \in C\sslash G$.  We call a connect sum of copies of $S^2_{\overline{c}}$, for possibly varying $c$, a \emph{puncture stabilizing sphere}.

Note that if $(T,\phi) \in R_{g,v}$, then a $\overline{c}$-stabilization of $(T,\phi)$ has branch type
\[ v + \delta_{(\overline{c},+1)} + \delta_{(\overline{c},-1)} \in \Z_{\geq 0}^{C\sslash G \times \{+1,-1\} } \]
where $\delta_{(\overline{c},\pm 1)}$ is the delta function on $(\overline{c},\pm 1)$.  In particular, a puncture stabilization never has positive trivialization.  (This fact is responsible for the ``dilation" map we introduce in the next subsection.)

Finally, we say two $C$-branched $G$-covers are \emph{stably equivalent} if they are equivalent after applying some sequence of handle and puncture stabilizations to each of them.

\begin{proposition}
Suppose $(T,\phi) \in R_{g,v}$ and $(S,\psi) \in R_{g',w}$ are connected $C$-branched $G$-covers of oriented surfaces.  Then $\sch_C(T,\phi) = \sch_C(S,\psi)$ if and only if $(T,\phi)$ and $(S,\psi)$ are stably equivalent.  

Suppose, moreover, that $\langle C \rangle = G$ and $g=g'$.  Then handle stabilization is unnecessary.  That is, $\sch_C(T,\phi) = \sch_C(S,\psi)$ if and only if $(T,\phi)$ and $(S,\psi)$ are puncture-stabilization equivalent.
\label{p:stable}
\end{proposition}

Livingston proved this result in the unbranched case for which $C$ is empty \cite{Livingston:stabilizing}.  Moreover, he considers the question of extending his results to the branched case, but without employing puncture stabilization.  He even gives an example of two unbranched covers that are bordant as branched covers, but not as unbranched covers.  Our proposition seems to be the appropriate generalization of Livingston's result to the (framed) branched case.

\begin{proof}

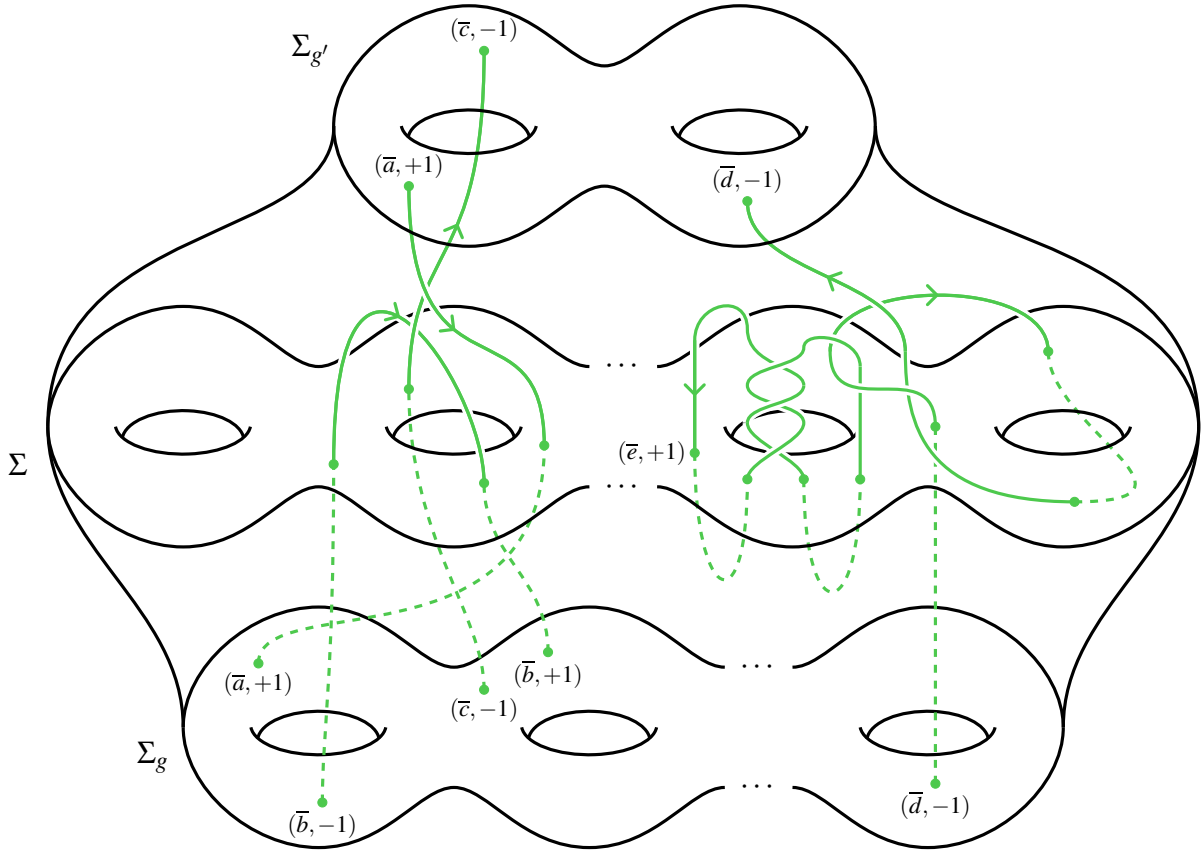
\begin{figure*}
\begin{tikzpicture}[decoration={markings,
    mark=at position 0.5 with {\arrow{angle 90}}}]
\begin{scope}[yscale=0.8,xscale=0.9]
\draw[very thick] (0,0) .. controls (0,2) and (-2,3) .. (-2,5).. controls (-2,8.5) and (2.225,8) .. (2.225,10);
\draw[very thick] (13,0) .. controls (13,2) and (15,3) .. (15,5) .. controls (15,8.5) and (10.225,8) .. (10.225,10);
\draw[very thick] (0,0) .. controls (0,-1) and (1,-2) .. (2,-2) .. controls (3,-2) and (3.5,-1) .. (4,-1) .. controls (4.5,-1) and (5,-2) .. (6,-2) .. controls (7,-2) and (7.5,-1) .. (8,-1);
\draw[very thick] (9,-1) .. controls (9.5,-1) and (10,-2) .. (11,-2) .. controls (12,-2) and (13,-1) .. (13,0);
\draw (8.5,-1) node {\large $\cdots$};
\draw[very thick] (1,0) arc [x radius = 1, y radius = .5, start angle = 185, end angle = 355];
\draw[very thick] (1.1,-.19) arc [x radius = .9, y radius = .5, start angle = 175, end angle = 5];
\begin{scope}[xshift=4cm]
\draw[very thick] (1,0) arc [x radius = 1, y radius = .5, start angle = 185, end angle = 355];
\draw[very thick] (1.1,-.19) arc [x radius = .9, y radius = .5, start angle = 175, end angle = 5];
\end{scope}
\begin{scope}[xshift=9cm]
\draw[very thick] (1,0) arc [x radius = 1, y radius = .5, start angle = 185, end angle = 355];
\draw[very thick] (1.1,-.19) arc [x radius = .9, y radius = .5, start angle = 175, end angle = 5];
\end{scope}
\begin{scope}[yscale=-1]
\draw[very thick] (0,0) .. controls (0,-1) and (1,-2) .. (2,-2) .. controls (3,-2) and (3.5,-1) .. (4,-1) .. controls (4.5,-1) and (5,-2) .. (6,-2) .. controls (7,-2) and (7.5,-1) .. (8,-1);
\draw[very thick] (9,-1) .. controls (9.5,-1) and (10,-2) .. (11,-2) .. controls (12,-2) and (13,-1) .. (13,0);
\draw (8.5,-1) node {\large $\cdots$};
\end{scope}
\end{scope}

\coordinate (ga1) at (1,.85);
\coordinate (gb1) at (1.85,-1);
\coordinate (gb2) at (4.85,1);
\coordinate (gc1) at (4.0,.5);
\coordinate (gd1) at (10,-.75);

\coordinate (mb1) at (2,3.5);
\coordinate (mc1) at (3,4.5);
\coordinate (mb2) at (4,3.25);
\coordinate (ma1) at (4.8,3.75);
\coordinate (me1) at (6.8,3.65);
\coordinate (me2) at (7.5,3.3);
\coordinate (me3) at (8.25,3.3);
\coordinate (me4) at (9,3.3);
\coordinate (md1) at (10,4);
\coordinate (md2) at (11.5,5);
\coordinate (md3) at (11.85,3);

\coordinate (ha1) at (3,7.2);
\coordinate (hc1) at (4,9);
\coordinate (hd1) at (7.5,7);

\draw[medgreen,dashed,very thick] (gb2) .. controls ($(gb2)+(0,1)$) and ($(mb2)-(0,1)$) .. (mb2);
\draw[medgreen,dashed,very thick] (ga1) .. controls ($(ga1)+(0,1)$) and ($(ma1)-(0,3)$) .. (ma1);
\draw[medgreen,dashed,very thick] (gb1) .. controls ($(gb1)+(0,1)$) and ($(mb1)-(0,3)$) .. (mb1);
\draw[medgreen,dashed,very thick] (gc1) .. controls ($(gc1)+(0,1)$) and ($(mc1)-(0,2)$) .. (mc1);
\draw[medgreen,dashed,very thick] (gd1) .. controls ($(gd1)+(0,1)$) and ($(md1)-(0,1)$) .. (md1);
\draw[medgreen,dashed,very thick] (me1) .. controls ($(me1)-(0,2)$) and ($(me2)-(0,2)$) .. (me2);
\draw[medgreen,dashed,very thick] (me3) .. controls ($(me3)-(0,2)$) and ($(me4)-(0,2)$) .. (me4);
\draw[medgreen,dashed,very thick] (md2) .. controls ($(md2)-(0,1)$) and ($(md3)-(-2,0)$) .. (md3);

\begin{scope}[yshift=4cm,yscale=0.8,xscale=0.9,xshift=-2cm]
\draw[very thick] (0,0) .. controls (0,-1) and (1,-2) .. (2,-2) .. controls (3,-2) and (3.5,-1) .. (4,-1) .. controls (4.5,-1) and (5,-2) .. (6,-2) .. controls (7,-2) and (7.5,-1) .. (8,-1);
\draw[very thick] (9,-1) .. controls (9.5,-1) and (10,-2) .. (11,-2) .. controls (12,-2) and (12.5,-1) .. (13,-1) .. controls (13.5,-1) and (14,-2) .. (15,-2) .. controls (16,-2) and (17,-1) .. (17,0);
\draw (8.5,-1) node {\large $\cdots$};
\draw[very thick] (1,0) arc [x radius = 1, y radius = .5, start angle = 185, end angle = 355];
\draw[very thick] (1.1,-.19) arc [x radius = .9, y radius = .5, start angle = 175, end angle = 5];
\begin{scope}[xshift=4cm]
\draw[very thick] (1,0) arc [x radius = 1, y radius = .5, start angle = 185, end angle = 355];
\draw[very thick] (1.1,-.19) arc [x radius = .9, y radius = .5, start angle = 175, end angle = 5];
\end{scope}
\begin{scope}[xshift=9cm]
\draw[very thick] (1,0) arc [x radius = 1, y radius = .5, start angle = 185, end angle = 355];
\draw[very thick] (1.1,-.19) arc [x radius = .9, y radius = .5, start angle = 175, end angle = 5];
\end{scope}
\begin{scope}[xshift=13cm]
\draw[very thick] (1,0) arc [x radius = 1, y radius = .5, start angle = 185, end angle = 355];
\draw[very thick] (1.1,-.19) arc [x radius = .9, y radius = .5, start angle = 175, end angle = 5];
\end{scope}
\begin{scope}[yscale=-1]
\draw[very thick] (0,0) .. controls (0,-1) and (1,-2) .. (2,-2) .. controls (3,-2) and (3.5,-1) .. (4,-1) .. controls (4.5,-1) and (5,-2) .. (6,-2) .. controls (7,-2) and (7.5,-1) .. (8,-1);
\draw[very thick] (9,-1) .. controls (9.5,-1) and (10,-2) .. (11,-2) .. controls (12,-2) and (12.5,-1) .. (13,-1) .. controls (13.5,-1) and (14,-2) .. (15,-2) .. controls (16,-2) and (17,-1) .. (17,0);
\draw (8.5,-1) node {\large $\cdots$};
\end{scope}
\end{scope}

\draw[white,double=medgreen,very thick,double distance=1.2] (mb1) .. controls ($(mb1)+(0,4)$) and ($(mb2)+(0,1.33)$) .. (mb2);
\draw[medgreen,very thick,postaction={decorate}] (mb1) .. controls ($(mb1)+(0,4)$) and ($(mb2)+(0,1.33)$) .. (mb2);
\draw[white,double=medgreen,very thick,double distance=1.2] (mc1) .. controls ($(mc1)+(0,2)$) and ($(hc1)-(0,3)$) .. (hc1);
\draw[medgreen,very thick,postaction={decorate}] (mc1) .. controls ($(mc1)+(0,2)$) and ($(hc1)-(0,3)$) .. (hc1);
\draw[white,double=medgreen,very thick,double distance=1.2] (ha1) .. controls ($(ha1)-(0,3)$) and ($(ma1)+(0,2)$) .. (ma1);
\draw[medgreen,very thick,postaction={decorate}] (ha1) .. controls ($(ha1)-(0,3)$) and ($(ma1)+(0,2)$) .. (ma1);

\draw[white,double=medgreen,very thick,double distance=1.2] (me4) -- ($(me4)+(0,1.5)$);

\draw[white,double=medgreen,very thick,double distance=1.2] (md3) .. controls ($(md3)+(-2,0)$) and ($(9.6,5)-(0,1)$) .. (9.6,5);
\draw[white,double=medgreen,very thick,double distance=1.2] (md1) .. controls ($(md1)+(0,1)$) and ($(8.6,5)-(0,1)$) .. (8.6,5) .. controls (8.6,6) and ($(md2)+(0,1)$) .. (md2);
\draw[medgreen,very thick,postaction={decorate}] (8.6,5) .. controls (8.6,6) and ($(md2)+(0,1)$) .. (md2);
\draw[white,double=medgreen,very thick,double distance=1.2] (9.6,5) .. controls (9.6,6) and ($(hd1)-(0,1)$) .. (hd1);
\draw[medgreen,very thick,postaction={decorate}] (9.6,5) .. controls (9.6,6) and ($(hd1)-(0,1)$) .. (hd1);

\draw[white,double=medgreen,very thick,double distance=1.2] (me3) .. controls ($(me3)+(0,.25)$) and ($(me2)+(0,.5)$) .. ($(me2)+(0,.75)$);
\draw[white,double=medgreen,very thick,double distance=1.2] ($(me3)+(0,1.25)$) .. controls ($(me3)+(0,1.5)$) and ($(me2)+(0,1.5)$) .. ($(me2)+(0,2)$) .. controls ($(me2)+(0,2.5)$) and ($(me1)+(0,2)$) .. ($(me1)+(0,1.5)$) -- (me1);
\draw[medgreen,very thick,postaction={decorate}] ($(me1)+(0,1.5)$) -- (me1);
\draw[white,double=medgreen,very thick,double distance=1.2] ($(me4)+(0,1.5)$) .. controls ($(me4)+(0,1.9)$) and ($(me3)+(0,2)$) .. ($(me3)+(0,1.75)$) .. controls ($(me3)+(0,1.5)$) and ($(me2)+(0,1.5)$)  .. ($(me2)+(0,1.25)$) .. controls ($(me2)+(0,1)$) and ($(me3)+(0,1)$) .. ($(me3)+(0,.75)$) .. controls ($(me3)+(0,.5)$) and ($(me2)+(0,.25)$) .. (me2);
\draw[white,double=medgreen,very thick,double distance=1.2] ($(me2)+(0,.75)$) .. controls ($(me2)+(0,1)$) and ($(me3)+(0,1)$) .. ($(me3)+(0,1.25)$);

\draw[left] (me1) node {$(\overline{e},+1)$};

\draw[medgreen,fill=medgreen] (ga1) circle[radius=.06];
\draw[below] (ga1) node {$(\overline{a},+1)$};
\draw[medgreen,fill=medgreen] (gb1) circle[radius=.06];
\draw[below] (gb1) node {$(\overline{b},-1)$};
\draw[medgreen,fill=medgreen] (gb2) circle[radius=.06];
\draw[below] (gb2) node {$(\overline{b},+1)$};
\draw[medgreen,fill=medgreen] (gc1) circle[radius=.06];
\draw[below] (gc1) node {$(\overline{c},-1)$};
\draw[medgreen,fill=medgreen] (gd1) circle[radius=.06];
\draw[below] (gd1) node {$(\overline{d},-1)$};
\draw (-.4,-.4) node {\large $\Sigma_g$};

\draw[medgreen,fill=medgreen] (mb1) circle[radius=.06];
\draw[medgreen,fill=medgreen] (mc1) circle[radius=.06];
\draw[medgreen,fill=medgreen] (mb2) circle[radius=.06];
\draw[medgreen,fill=medgreen] (ma1) circle[radius=.06];
\draw[medgreen,fill=medgreen] (me1) circle[radius=.06];
\draw[medgreen,fill=medgreen] (me2) circle[radius=.06];
\draw[medgreen,fill=medgreen] (me3) circle[radius=.06];
\draw[medgreen,fill=medgreen] (me4) circle[radius=.06];
\draw[medgreen,fill=medgreen] (md1) circle[radius=.06];
\draw[medgreen,fill=medgreen] (md2) circle[radius=.06];
\draw[medgreen,fill=medgreen] (md3) circle[radius=.06];
\draw (-2.2,3.5) node {\large $\Sigma$};

\draw[medgreen,fill=medgreen] (ha1) circle [radius=.06];
\draw[above] (ha1) node {$(\overline{a},+1)$};
\draw[medgreen,fill=medgreen] (hc1) circle [radius=.06];
\draw[above] (hc1) node {$(\overline{c},-1)$};
\draw[medgreen,fill=medgreen] (hd1) circle [radius=.06];
\draw[above] (hd1) node {$(\overline{d},-1)$};
\draw (1.7,9) node {\large $\Sigma_{g'}$};

\begin{scope}[xshift=2cm,yshift=8cm,yscale=0.8,xscale=0.9]
\draw[very thick] (0,0) .. controls (0,-1) and (1,-2) .. (2,-2) .. controls (3,-2) and (3.5,-1) .. (4,-1) .. controls (4.5,-1) and (5,-2) .. (6,-2) .. controls (7,-2) and (8,-1) .. (8,0);
\draw[very thick] (1,0) arc [x radius = 1, y radius = .5, start angle = 185, end angle = 355];
\draw[very thick] (1.1,-.19) arc [x radius = .9, y radius = .5, start angle = 175, end angle = 5];
\begin{scope}[xshift=4cm]
\draw[very thick] (1,0) arc [x radius = 1, y radius = .5, start angle = 185, end angle = 355];
\draw[very thick] (1.1,-.19) arc [x radius = .9, y radius = .5, start angle = 175, end angle = 5];
\end{scope}
\begin{scope}[yscale=-1]
\draw[very thick] (0,0) .. controls (0,-1) and (1,-2) .. (2,-2) .. controls (3,-2) and (3.5,-1) .. (4,-1) .. controls (4.5,-1) and (5,-2) .. (6,-2) .. controls (7,-2) and (8,-1) .. (8,0);
\end{scope}
\end{scope}

\end{tikzpicture}
\caption{A cobordism $M$ between the branched cover $(T,\phi)$ of $\Sigma_g$ and the branched cover $(S,\psi)$ of $\Sigma_{g'}$ provides a mutual stabilization of the two covers.  The branch locus $K \subseteq M$ is drawn in green.  The dashed green arcs denote $K_1$, which is the part of $K$ between $\Sigma_g$ and the relative Heegaard splitting $\Sigma$.  The solid green arcs denote $K_2$, which is the part of $K$ between $\Sigma$ and $\Sigma_{g'}$.  The labels $a,b,c,d,e$ denote elements in $C$.   To avoid overcrowding this schematic, we have not drawn a Heegaard diagram on $\Sigma$.}
\label{f:cob}
\end{figure*}

As in Livingston's proof, the main idea is to use the fact that integral homology and oriented cobordism are the same in dimension 2.  Thus, if the homology classes $\sch_C(T,\phi)$ and $\sch_C(S,\psi)$ in $H_2(BG_C)$ are equal, then there is a branched cover of some oriented 3-manifold $M$ with boundary the union of $(T,\phi)$ and $(S,\psi)$.  Let $K \subseteq M$ be the branch locus of this cover.

Construct a relative handle decomposition of $M$ and arrange so that all 1-handles are attached before 2-handles, and so that the attaching maps of the handles avoid the branch loci on $\Sigma_g$ and $\Sigma_{g'}$.  Because $(T,\phi)$ is connected, we can slide the attaching maps for the 1-handles in the complement of the branch locus to guarantee that the $G$-cover over each of these new handles is trivial.  We can do the same thing for the 2-handles because $(S,\psi)$ is connected.  See \cite{Livingston:stabilizing} for details.  Let $\Sigma$ be the resulting relative Heegaard surface, which is the intersection of the compression bodies
\[\begin{aligned}
H_1 &\defeq \Sigma_g \cup \{\text{1-handles}\}, \\
H_2 &\defeq \Sigma_h \cup \{\text{2-handles}\}.
\end{aligned}\]
Isotope $K$ so that it is transverse to $\Sigma$, all maxima of $K$ are on the top side of $\Sigma$, and all minima are on the bottom side.  The latter two conditions can be understood as saying that $\Sigma$ is a bridge surface for the tangle $K \subseteq M$.  The $C$-branched $G$-cover of $M$ restricts to a $C$-branched $G$-cover over $\Sigma$.  See Figure \ref{f:cob}.

We claim that this cover of $\Sigma$ is a stabilization of both $(T,\phi)$ and $(S,\psi)$.  Let
\[\begin{aligned}
K_1 &\defeq K \cap H_1, \\
K_2 &\defeq K \cap H_2.
\end{aligned}\]
So $K_1$ is the portion of $K$ that lies between $\Sigma_g$ and $\Sigma$, while $K_2$ is the portion of $K$ that lies between $\Sigma$ and $\Sigma_{g'}$.  Let $\alpha_1,\dots,\alpha_k \subseteq H_1$ be the components of $K_1$ without endpoints on $\Sigma_g$.  The conditions on the maxima and minima of $K$ guarantee that $K_1$ can be isotoped inside $H_1$ so that the arcs $\alpha_1,\dots,\alpha_k$ lie inside $\Sigma$ simultaneously.  Equivalently, there is a collection of disjoint disks $D_1,\dots,D_k$ in $H_1$ such that $\alpha_i \subseteq \partial D_i$ and $\partial D_i \setminus \alpha_i \subseteq \Sigma$.  If we restrict the cover over $\Sigma$ to a regular neighborhood of $\partial D_i \setminus \alpha_i^\circ$ in $\Sigma$, we see a summand equivalent to $S^2_{\overline{c}}$ for some $\overline{c} \in C\sslash G$.  We already slid the 1-handles to ensure the cover over them is trivial, so we conclude that $\Sigma$ is a stabilization of $(T,\phi)$.  This same argument with $\Sigma_{g'}$, $H_2$ and $K_2$ used in place of $\Sigma_g$, $H_1$ and $K_1$ shows that the cover over $\Sigma$ is a stabilization of $(S,\psi)$.  Thus $(T,\phi)$ and $(S,\psi)$ are stably equivalent.

Conversely, suppose $(T,\phi)$ and $(S,\psi)$ are stably equivalent.  It follows from Figures \ref{f:trivializations} and \ref{f:torus} that $S^2_{\overline{c}}$ and the trivial cover of $S^1 \times S^1$ are both null-cobordant.  Since homology and oriented bordism are the same in dimension 2, the branched Schur invariants of these covers are both trivial.  The Schur invariant of a connect-sum is the sum of the Schur invariants.  Hence puncture stabilization and handle stabilization both preserve the Schur invariant, and so $\sch_C(T,\phi) = \sch_C(S,\psi)$.

Finally, suppose $\langle C \rangle = G$, $g=g'$, and $\sch_C(T,\phi) = \sch_C(S,\psi)$.  Lemma \ref{l:pi1} says $\pi_1(BG_C)$ is trivial, so the Hopf-Whitney classification \cite[Corollary V.6.19]{Whitehead:homotopyGTM} and the universal coefficient theorem imply that homotopy classes of maps $\Sigma_g \to BG_C$ are determined by their induced maps $H_2(\Sigma_g) \to H_2(BG_C)$, which of course are determined by the values taken on $[\Sigma_g]$, \emph{i.e.} the Schur invariant.  Thus, $(T,\phi)$ and $(S,\psi)$ induce homotopic maps to $BG_C$.  As in the proof of Theorem \ref{th:brand}, this implies $(T,\phi)$ and $(S,\psi)$ are cobordant via a cylinder $M = \Sigma_g \times [0,1]$.  As before, isotope the branch locus $K \subset M$ so that it is in bridge position with respect to the surface $\Sigma = \Sigma_g \times \{1/2\}$.  The restriction of the cover over $M$ to $\Sigma$ is a mutual puncture stabilization of $(T,\phi)$ and $(S,\psi)$.
\end{proof}

We remark that Proposition \ref{p:stable} is true even if $G$ is infinite.  However, all of our other results in this section use finiteness of $G$ in some way.

Any two $\overline{c}$-stabilizations of a given connected cover are equivalent.  Thus, $\overline{c}$-stabilization yields a well-defined map
\[ p_{\overline{c}}: R_{g,v}/\MCG_*(\Sigma_{g,n}) \to R_{g,v + \delta_{(\overline{c},+1)} + \delta_{(\overline{c},-1)}} / \MCG_*(\Sigma_{g,n+2}). \]
Similarly, any two handle stabilizations are equivalent, so handle stabilization yields a map
\[ h: R_{g,v}/\MCG_*(\Sigma_{g,n}) \to R_{g+1,v}/\MCG_*(\Sigma_{g+1,n}). \]
Since two elements of $R_{g,v}$ in the same $\MCG_*(\Sigma_{g,n})$ orbit are equivalent as $C$-branched $G$-covers, they must have the same Schur invariant.  Thus, the Schur invariant yields a map
\[ \sch_C: R_{g,v}/\MCG_*(\Sigma_{g,n}) \to H_2(BG_C). \]
All of these maps commute.  More precisely, see the commuting diagram in Figure \ref{f:commute}.  A similar diagram commutes where, instead of involving a puncture stabilization and a handle stabilization, there are two puncture stabilizations of different types.

\begin{figure*}[t]
\begin{tikzcd}
 & R_{g,v + \delta_{(\overline{c},+1)} + \delta_{(\overline{c},-1)}} / \MCG_*(\Sigma_{g,n+2})  \arrow{dr}{h} \arrow{d}{\sch_C}& \\
R_{g,v}/\MCG_*(\Sigma_{g,n}) \arrow{r}{\sch_C} \arrow{ur}{p_{\overline{c}}} \arrow{dr}{h} & H_2(BG_C)  & R_{g+1,v + \delta_{(\overline{c},+1)} + \delta_{(\overline{c},-1)}} / \MCG_*(\Sigma_{g+1,n+2}) \arrow{l}{\sch_C}  \\
 & R_{g+1,v}/\MCG_*(\Sigma_{g+1,n}) \arrow{ur}{p_{\overline{c}}} \arrow{u}{\sch_C} & 
\end{tikzcd}
\caption{The diagram commutes.}
\label{f:commute}
\end{figure*}

The next two lemmas show that whenever $g$ and $v$ are large enough, everything is a stabilization.

\begin{lemma}
If $g>|G|$, then $h$ is surjective.
\label{l:handle}
\end{lemma}

\begin{proof}
The proof is exactly the same as the proof of Proposition 6.16 in Dunfield and Thurston's work \cite{DunfieldThurston:random}.  In particular, the branch locus plays no role.

Let $(T,\phi) \in R_{g,v}$.  We want to find a torus with one boundary component inside $\Sigma_{g,n}$ where the cover $(T,\phi)$ is trivial over it.  Let $a_1,b_1,\dots,a_g,b_g$ be the elements of $\pi_1(\Sigma_{g,n})$ indicated in Figure \ref{f:wheel}, and let $w_i = a_i \cdots a_2 \cdot a_1$ for each $i=1,\dots,g$.  Since $g>|G|$, the pigeonhole principle guarantees $\phi(w_i) = \phi(w_j)$ for some $i<j$, hence $\phi(a_j\cdots a_{i+1}) = 1$.  Thus, $a_j\cdots a_{i+1}$ is a non-separating simple closed curve on $\Sigma_g$ in the kernel of $\phi$.  Let $c_1,\dots,c_k$ be a maximal disjoint collection of such curves.  After applying some equivalence, we can make $c_1=b_1,\dots,c_k=b_k$.  The same argument as before shows that some $w=a_j\cdots a_{i+1}$ is in the kernel of $\phi$.  Since the collection $c_1,\dots,c_k$ is maximal, $w$ must intersect one of the curves $c_l$.  Since $w$ and $c_l$ intersect in exactly one point, a regular neighborhood of the two curves is a torus with boundary.  By construction, $\phi$ is trivial on this torus, which shows $(T,\phi)$ is a handle stabilization.
\end{proof}

\begin{figure}[t]
\begin{tikzpicture}[decoration={markings,
    mark=at position 0.55 with {\arrow{angle 90}}},scale=0.9]
\clip (-4.9,-4.2) rectangle (4.9, 4.9);
\foreach \i in {0,2,...,6} {
\coordinate (a\i) at (36*\i:2cm);
\coordinate (b\i) at (36*\i+36:2cm);
\draw[very thick] (0,0) +(36*\i+36:2cm) arc [radius=2cm,delta angle=36,start angle=36*\i+36];
\draw[very thick] (a\i) .. controls ($3*(a\i)$) and ($3*(b\i)$) .. (b\i);
\draw[very thick] (36*\i+18:2.5cm) .. controls (36*\i+9:3.25cm) .. (36*\i+18:4cm);
\draw[very thick] (36*\i+16.5:2.6cm) .. controls (36*\i+27:3.25cm) .. (36*\i+17:3.9cm);
\draw[thick, blue,postaction={decorate}] (0,0) .. controls (36*\i+9:1.5cm) and (36*\i+6:2.5cm) .. (36*\i+6:3.25cm) .. controls (36*\i+6:4cm) and (36*\i+14:4.5cm) .. (36*\i+18:4.5cm) .. controls (36*\i+22:4.5cm) and (36*\i+30:4cm) .. (36*\i+30:3.25cm) .. controls  (36*\i+30:2.5cm) and (36*\i+27:1.5cm) .. (0,0);
\draw[thick,red] (0,0) .. controls (36*\i+5:2cm) and (36*\i+5:2.5cm) .. (36*\i:2.5cm);
\draw[thick,dashed,red] (0,0) (36*\i:2.5cm) -- (36*\i+13:2.9cm);
\draw[thick,red,postaction={decorate}] (36*\i+13:2.9cm) .. controls (36*\i+9:2.8cm) and (36*\i+15:2cm) .. (0,0);
}
\foreach \i in {1,2,3} {
\draw (72*\i+23:3.9cm) node {$a_{\i}$};
\draw (72*\i+18:2.2cm) node {$b_{\i}$};
}
\foreach \i in {0} {
\draw (72*\i+23:3.9cm) node {$a_{g}$};
\draw (72*\i+18:2.2cm) node {$b_{g}$};
}
\draw[very thick] (0,0) +(0:2cm) arc [radius=2cm,delta angle=-36,start angle=0];
\draw[fill,medgreen] (0,0) +(-40:1cm) circle [radius=0.03];
\draw[fill,medgreen] (0,0) +(-68:1cm) circle [radius=0.03];
\draw[fill,medgreen] (0,0) +(-76:1cm) circle [radius=0.03];
\draw[fill,medgreen] (0,0) +(-48:1cm) circle [radius=0.01];
\draw[fill,medgreen] (0,0) +(-54:1cm) circle [radius=0.01];
\draw[fill,medgreen] (0,0) +(-60:1cm) circle [radius=0.01];
\draw[fill] (0,0) +(-48:2cm) circle [radius=0.03];
\draw[fill] (0,0) +(-54:2cm) circle [radius=0.03];
\draw[fill] (0,0) +(-60:2cm) circle [radius=0.03];
\draw[fill] (0,0) circle [radius=0.04];
\end{tikzpicture}
\caption{The elements of $\pi_1(\Sigma_{g,n})$ used in the proof of Lemma \ref{l:handle}.}
\label{f:wheel}
\end{figure}
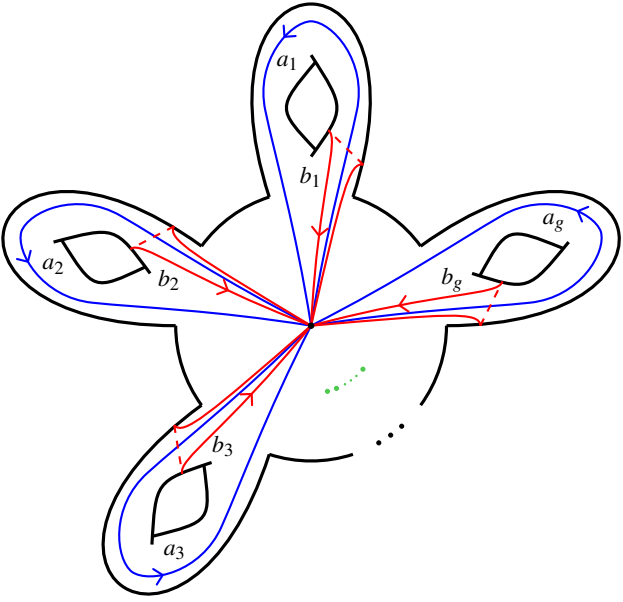

We say $v$ is \emph{larger} than $w$ if
\[ v(\overline{c},o) > w(\overline{c},o) \]
for all $(\overline{c},o) \in C\sslash G \times \{+1,-1\}$.  For each conjugacy class $\overline{c}$ in  $C\sslash G$, define a special branch type
\[ U^{\overline{c}} \in \mathbb{Z}_{\geq 0}^{C\sslash G \times \{+1,-1\}}\] that consists of the single element $(\overline{c},+1)$ with multiplicity $|\overline{c}| \cdot |\Inn_{\overline{c}}(c)|$, where $|\Inn_{\overline{c}}(c)|$ is the order of $c$ acting as an inner automorphism on the conjugacy class $\overline{c}$.  That is,
\[ U^{\overline{c}} = |\overline{c}| \cdot |\Inn_{\overline{c}}(c)| \cdot \delta_{(\overline{c},1)}. \]

\begin{lemma}
If $v$ is larger than $U^{\overline{c}}$, then the puncture stabilization map $p_{\overline{c}}$ is surjective.
\label{l:puncture}
\end{lemma}

\begin{proof}
This lemma and its proof extend Lemma 4 of the appendix in Fried and Volklein's paper \cite{FriedVolklein:galois}.  

Let $(T,\phi)$ be an element of $R_{g,v + \delta_{(\overline{c},+1)} + \delta_{(\overline{c},-1)}}$.  Pick a disk $D$ in $\Sigma_{g,n+2}$ containing the branch set $K$.  Pick a basepoint on the disk, and consider a generating set for $\pi_1(D\setminus K)$ consisting of small loops around the branch points, oriented according to $T$.  Apply a braid to arrange so that a branch point of type $(\overline{c},-1)$ is to the right of all the other branch points.  The condition on $v$ and the pigeonhole principle guarantee that there is some $c' \in \overline{c}$ that appears as the monodromy around $|\Inn_{\overline{c}}(c)|$ many of the positively-oriented branch points.  Move all of these branch points to the left of the disk, being sure to braid them over the rest of the branch points, not under.  This conjugates the monodromies of the branch points that pass underneath, but preserves the monodromies of the $|\Inn_{\overline{c}}(c)|$ points passing over.  The resulting monodromies, arranged from left to right, look as follows:
\[ c_1, \ c_1, \ c_1, \dots \ c_1, *, \ *, \dots *, c_2^{-1},\]
where there are $|\Inn_{\overline{c}}(c)|$ copies of $c_1$ on the left, $c_2 \in \overline{c}$, and the asterisks $*$ denote arbitrary elements of $C$.  Choose a smaller disk $D'$ that contains only the $|\Inn_{\overline{c}}(c)|$ branch points on the left.  Since $(T,\phi)$ is connected, there is some loop on $\Sigma_{g,n+2}$ around which we can slide $D'$ so that the result is
\[ c_2, \ c_2, \ c_2, \dots \ c_2, *, \ *, \dots *, c_2^{-1}.\]
In particular, this sliding of $D'$ does not change $c_2^{-1}$ because any time $c_2^{-1}$ gets conjugated by a loop around a puncture in $D'$, it gets conjugated by all $|\Inn_{\overline{c}}(c)|$ of them.  Now apply another braid to arrange so that a point of type $c_2$ is adjacent to the point of type $c_2^{-1}$.  This shows that $\phi$ is equivalent to a $\overline{c}$-stabilization.
\end{proof}

\subsection{Dilation replacement}
\label{ss:dilation}
In this subsection, we introduce a method called \emph{dilation} for converting a branched cover with negatively trivialized branch points into a cover with a positive trivialization.  The dilation of branching data is a linear map
\[ \begin{aligned}
\dil: \mathbb{Z}_{\ge 0}^{C\sslash G \times \{+1,-1\}} &\to \mathbb{Z}_{\ge 0}^{C\sslash G} \\
w &\mapsto w_{\dil}
\end{aligned} \]
where $w_{\dil}$ is defined by
\[ w_{\dil}(\overline{c}) \defeq w(\overline{c},+1) + (\ord(c)-1) \cdot w(\overline{c},-1). \]
Here $\ord(c)$ is the order of the group element $c$.  If $w$ has cardinality $n$, then $w_{\dil}$ has cardinality $n+N$ where
\[ N = \sum_{\overline{c} \in C\sslash G} (\ord(c)-2)w(\overline{c},-1). \]

Extend this notion of dilation to elements of $R_{g,w}$ via the map
\[ \begin{aligned}
\dil: R_{g,w} &\to R_{g,w_{\dil}} \\
(T,\phi) &\mapsto (+1,\phi_{\dil}) = \phi_{\dil}
\end{aligned} \]
where $+1$ denotes the positive trivialization and $\phi_{\dil}$ is defined by replacing each puncture of type $(\overline{c},-1)$ with $\ord(c)-1$ punctures of type $(\overline{c},+1)$ as in Figure \ref{f:dilation}.  Note that this map $\dil$ depends on a choice of generating set of $\pi_1(\Sigma_{g,n})$; we use any standard set of generators that contains those indicated in the figure.  However, this choice is irrelevant for our purposes, because of our next observation.

\begin{figure*}
\begin{tikzpicture}[decoration={markings,
    mark=at position 0.4 with {\arrow{angle 90}}}]
\draw[very thick] (0,0) circle [x radius=6, y radius = 3];
\draw[very thick, fill = white] (-5,0) .. controls (-5,5) and (-3,5) .. (-3,0) arc [start angle=0, end angle=-180,x radius=1, y radius = .5];
\draw[dashed] (-3,0) arc [start angle=0, end angle=180,x radius =1, y radius = .5];
\draw[very thick] (-4.5,1) arc [start angle = 190, end angle = 350, x radius = .5, y radius = .25];
\draw[very thick] (-4.4,.9) arc [start angle = 175, end angle = 5, x radius = .4, y radius = .2];
\draw[very thick] (-4.5,2.5) arc [start angle = 190, end angle = 350, x radius = .5, y radius = .25];
\draw[very thick] (-4.4,2.4) arc [start angle = 175, end angle = 5, x radius = .4, y radius = .2];
\draw (-4,1.75) node {\large $\vdots$};
\coordinate (c1) at (-1.5,0);
\coordinate (c2) at (-.5,0);
\coordinate (c3) at (.5,0);
\coordinate (c4) at (2.5,0);
\coordinate (c5) at (3.5,0);
\coordinate (c6) at (4.5,0);
\foreach \i in {1,2,...,6} {\fill[medgreen] (c\i) circle [radius=0.05];}
\coordinate (p) at (1.5,-2.5);
\fill[black] (p) circle [radius=0.08];
\foreach \i in {1,4,6} {
\draw[thick,rounded corners,postaction={decorate}] (p) -- ($(c\i)+(.3,0)$) -- ($(c\i)+(0,.3)$) -- ($(c\i)-(.3,0)$) -- (p);
}
\draw ($(c1)+(0,.5)$) node {\large $c_1$};
\draw ($(c2)+(0,.5)$) node {\large $c_2$};
\draw ($(c3)+(0,.5)$) node {\large $c_3$};
\draw ($(c4)+(0,.5)$) node {\large $c_{n-2}$};
\draw ($(c5)+(0,.5)$) node {\large $c_{n-1}$};
\draw ($(c6)+(0,.5)$) node {\large $c_n$};
\foreach \i in {2,3,5} {
\draw[thick,rounded corners,postaction={decorate}] (p) -- ($(c\i)-(.3,0)$)-- ($(c\i)+(0,.3)$) -- ($(c\i)+(.3,0)$)  --  (p);
}

\draw[medgreen] (1.5,0) node {\large $\cdots$};

\draw[|->,very thick] (0,-3.5)--(0,-4.5);
\draw(.5,-4) node {\large $\dil$};

\begin{scope}[yshift=-8cm]
\draw[very thick] (0,0) circle [x radius=6, y radius = 3];
\draw[very thick, fill = white] (-5,0) .. controls (-5,5) and (-3,5) .. (-3,0) arc [start angle=0, end angle=-180,x radius=1, y radius = .5];
\draw[dashed] (-3,0) arc [start angle=0, end angle=180,x radius =1, y radius = .5];
\draw[very thick] (-4.5,1) arc [start angle = 190, end angle = 350, x radius = .5, y radius = .25];
\draw[very thick] (-4.4,.9) arc [start angle = 175, end angle = 5, x radius = .4, y radius = .2];
\draw[very thick] (-4.5,2.5) arc [start angle = 190, end angle = 350, x radius = .5, y radius = .25];
\draw[very thick] (-4.4,2.4) arc [start angle = 175, end angle = 5, x radius = .4, y radius = .2];
\draw (-4,1.75) node {\large $\vdots$};
\coordinate (d1) at (-2,0);
\coordinate (d4) at (2.25,0);
\coordinate (d6) at (5,0);
\coordinate (q) at (1.5,-2.5);
\foreach \i in {1,4,6} {
\draw[thick,rounded corners,postaction={decorate}] (q) -- ($(d\i)+(.2,0)$) -- ($(d\i)+(0,.2)$) -- ($(d\i)-(.2,0)$) -- (q);
\fill[medgreen] (d\i) circle [radius=0.05];
}
\draw ($(d1)+(0,.5)$) node {\large $c_1$};
\draw ($(d4)+(0,.5)$) node {\large $c_{n-2}$};
\draw ($(d6)+(0,.5)$) node {\large $c_n$};
\fill[black] (q) circle [radius=0.08];

\coordinate (d21) at (-1.25,0);
\coordinate (d22) at (-.5,0);
\coordinate (d31) at (.25,0);
\coordinate (d32) at (1,0);
\coordinate (d51) at (3.25,0);
\coordinate (d52) at (4,0);
\foreach \i in {2,3,5} {\fill[medgreen] (d\i1) circle [radius=0.05];
\fill[medgreen] (d\i2) circle [radius=0.05];}
\draw[medgreen] (1.625,0) node {\large $\cdots$};
\draw[medgreen] (-.825,0) node {$\cdots$};
\draw[medgreen] (.65,0) node {$\cdots$};
\draw[medgreen] (3.6,0) node {$\cdots$};

\foreach \i in {2,3,5} {
\draw[thick,rounded corners,postaction={decorate}] (q) -- ($(d\i1)+(.2,0)$) -- ($(d\i1)+(0,.2)$)-- ($(d\i1)-(.2,0)$)  --  (q);
\draw[thick,rounded corners,postaction={decorate}] (q) -- ($(d\i2)+(.2,0)$) -- ($(d\i2)+(0,.2)$)-- ($(d\i2)-(.2,0)$)  --  (q);
}
\draw ($(d21)+(0,.5)$) node {\large $c_2$};
\draw ($(d22)+(0,.5)$) node {\large $c_2$};
\draw ($(d31)+(0,.5)$) node {\large $c_3$};
\draw ($(d32)+(0,.5)$) node {\large $c_3$};
\draw ($(d51)+(0,.5)$) node {$c_{n-1}$};
\draw ($(d52)+(0,.5)$) node {$c_{n-1}$};

\draw[decorate,decoration={brace,amplitude=.1cm}] ($(d21)+(0,.7)$) -- ($(d22)+(0,.7)$);
\draw (-.835,1) node {\tiny $\ord(c_2)-1$};
\draw[decorate,decoration={brace,amplitude=.1cm}] ($(d31)+(0,.7)$) -- ($(d32)+(0,.7)$);
\draw (.635,1) node {\tiny $\ord(c_3)-1$};
\draw[decorate,decoration={brace,amplitude=.1cm}] ($(d51)+(0,.7)$) -- ($(d52)+(0,.7)$);
\draw (3.6,1) node {\tiny $\ord(c_{n-1})-1$};

\end{scope}
\end{tikzpicture}
\caption{The dilation $\phi_{\dil}$ of $(T,\phi) \in R_{g,w}$.  We suppress the description of the monodromy around the handles, since the dilation map does not depend on it anyway.}
\label{f:dilation}
\end{figure*}
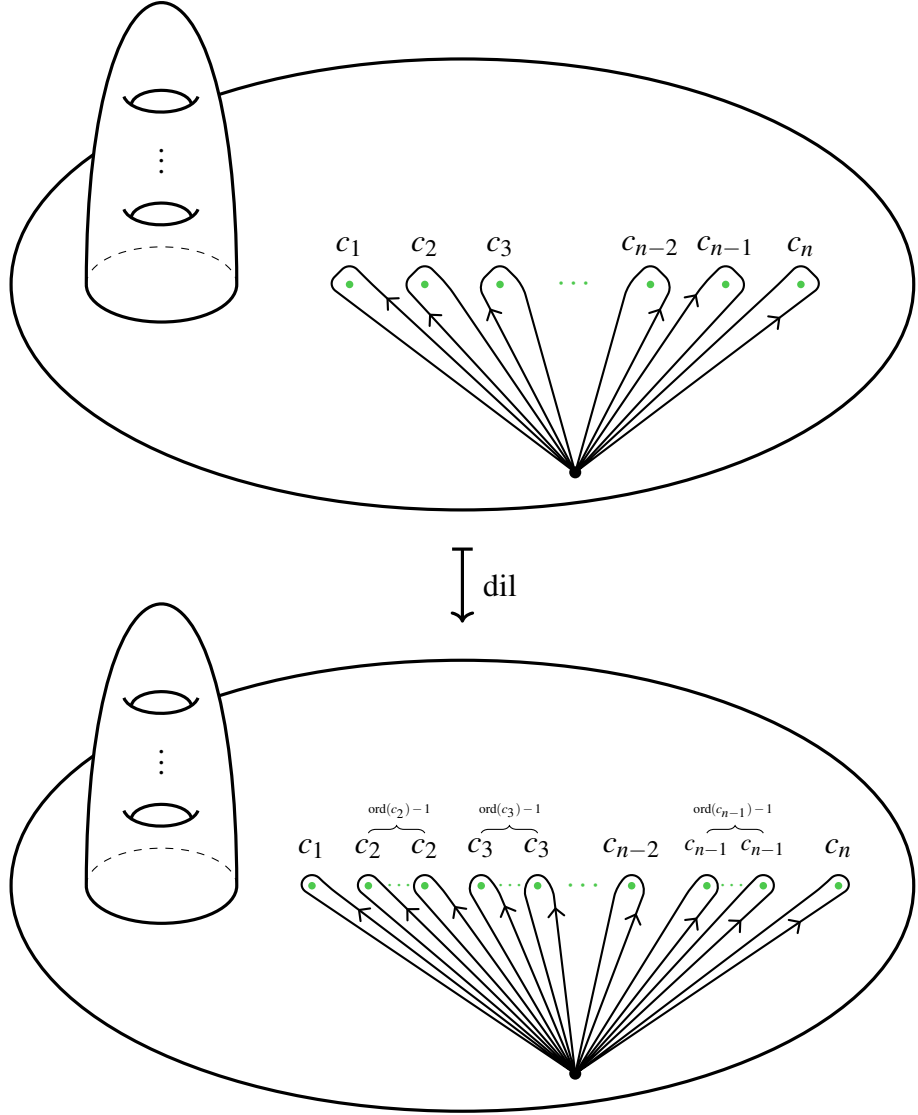

When a mapping class in $\MCG_*(\Sigma_{g,w})$ acts on an element $(T,\phi)$ in $R_{g,w}$, there is an induced action on $\phi_{dil}$ by banding together the dilated points.  Thus, $\dil$ descends to a well defined map
\[ \dil: R_{g,w}/\MCG_*(\Sigma_{g,n}) \to R_{g,w_{\dil}}/\MCG_*i(\Sigma_{g,n+N}). \]

\begin{lemma}
If $w$ satisfies the inequality
\[ w_{\dil}(\overline{c}) > |\overline{c}|\cdot (\ord(c)-1) \cdot w(\overline{c},-1)\]
for all $\overline{c} \in C\sslash G$, then
\[ \dil: R_{g,w}/\MCG_*(\Sigma_{g,n}) \to R_{g,w_{\dil}}/\MCG_*(\Sigma_{g,n+N}) \]
is surjective.
\label{l:dilation}
\end{lemma}

In particular, the condition of the lemma is satisfied if
\[ \]
We reiterate that $|\overline{c}|$ is the cardinality of the conjugacy class $\overline{c}$, while $\ord(c)$ is the order of the group element $c$.  Moreover, since
\[ w_{\dil}(\overline{c}) = w(\overline{c},+1) + (\ord(c)-1) \cdot w(\overline{c},-1), \]
the inequality in the lemma is equivalent to
\[ w(\overline{c},+1) > (|\overline{c}|-1)\cdot (\ord(c)-1) \cdot w(\overline{c},-1). \]
So the conditions of the lemma are stronger than saying $w$ is large enough.  Fortunately, we will later use this lemma in a ``backwards" way, \ie by picking $w_{\dil}$ first and then choosing $w$ so that it satisfies the conditions of the lemma.  See Lemma \ref{l:final}.

\begin{proof}[Proof of Lemma \ref{l:dilation}]
Let $\psi \in R_{g,w_{\dil}}$.  We must show $\psi$ is equivalent to $\phi_{\dil}$ for some $(T,\phi) \in R_{g,w}$.

For each $\overline{c}$, pick a disk $D_{\overline{c}} \subseteq \Sigma_g$ containing the basepoint and all of the branch points of type $\overline{c}$.  Arranged from left to right, the monodromies around these points reads
\[ c_1, \ c_2, \ c_3, \ \dots, \ c_{w_{\dil}(\overline{c})}, \]
where $c_1,c_2,\dots,c_{w_{\dil}(\overline{c})} \in \overline{c}$.
By iteratively applying the pigeonhole principle and a similar argument as in the proof of Lemma \ref{l:puncture}, we can find a braid that takes this to
\begin{center}
\begin{tikzpicture}
\coordinate (d11) at (0,0);
\coordinate (d12) at (1,0);
\coordinate (d21) at (2,0);
\coordinate (d22) at (3,0);
\coordinate (dk1) at (5,0);
\coordinate (dk2) at (6,0);
\coordinate (db1) at (7,0);
\coordinate (db2) at (8,0);
\draw (0.5,0) node {$\cdots$};
\draw (2.5,0) node {$\cdots$};
\draw (4,0) node {$\cdots$};
\draw (5.5,0) node {$\cdots$};
\draw (7.5,0) node {$\cdots$};
\draw (d11) node {$d_1$};
\draw (d12) node {$d_1$};
\draw (d21) node {$d_2$};
\draw (d22) node {$d_2$};
\draw (dk1) node {$d_k$};
\draw (dk2) node {$d_k$};
\draw (db1) node {$*$};
\draw (db2) node {$*$};
\draw[decorate,decoration={brace,amplitude=.1cm}] ($(d11)+(0,.2)$) -- ($(d12)+(0,.2)$);
\draw[decorate,decoration={brace,amplitude=.1cm}] ($(d21)+(0,.2)$) -- ($(d22)+(0,.2)$);
\draw[decorate,decoration={brace,amplitude=.1cm}] ($(dk1)+(0,.2)$) -- ($(dk2)+(0,.2)$);
\draw (0.5,.6) node {$\ord(c)-1$};
\draw (2.5,.6) node {$\ord(c)-1$};
\draw (5.5,.6) node {$\ord(c)-1$};
\end{tikzpicture}
\end{center}

\noindent where $k={w(\overline{c},-1)}$ and the asterisks $*$ denote various leftover elements in $\overline{c}$.  If we pick all of our disks $D_{\overline{c}}$ to be disjoint (except for the basepoint), then this shows $\psi$ is equivalent to a dilation.  Indeed, since the braids for each $\overline{c}$ are supported in their respective disks, they do not interfere with each other.
\end{proof}

Dilation does not preserve the Schur invariant.  However, we can compute the Schur invariant of a dilation in certain cases.  We require a definition: let $w$ and $v'$ be branching data in $\mathbb{Z}_{\ge 0}^{C\sslash G \times \{+1,-1\}}$.  We say $w$ is a \emph{stabilization} of $v'$ if
\[ (w-v')(\overline{c},+1) = (w-v')(\overline{c},-1) \]
for all $\overline{c} \in C\sslash G$.

\begin{lemma}
Let $w$ and $v'$ be branching data in $\mathbb{Z}_{\ge 0}^{C\sslash G \times \{+1,-1\}}$, and suppose $w$ is a stabilization of $v'$.  If $v'$ is larger than
\[ \sum_{\overline{c} \in C\sslash G} U^{\overline{c}}, \]
then there exists a homology class $\chi_{\dil(w-v')} \in H_2(BG_C)$ such that
\[ \sch_C(\dil(T,\phi)) = \sch_C(T,\phi) + \chi_{\dil(w-v')} \]
for all $(T,\phi) \in R_{g,w}$.
\label{l:dilhom}
\end{lemma} 

\begin{proof}
By Lemma \ref{l:puncture}, every element of $R_{g,w}$ is a stabilization of an element in $R_{g,v'}$.  That is, there exists some puncture stabilizing sphere, which we shall denote by $S^2_{w-v'}$, such that for all $(T,\phi) \in R_{g,w}$, $(T,\phi)$ is equivalent to the connect sum of $S^2_{w-v'}$ and some element $(S,\psi) \in R_{g,v'}$.  There exists another branched cover of the sphere, which we denote by $S^2_{\dil(w-v')}$, such that $\dil(T,\phi)$ is equivalent to the connect sum of $(S,\psi)$ and $S^2_{\dil(w-v')}$.  Let $\chi_{\dil(w-v')} = \sch_C(S^2_{\dil(w-v')})$.  Then
\[ \begin{aligned}
\sch_C(\dil(&T,\phi)) = \sch_C(S,\psi) + \sch_C(S^2_{\dil(w-v')}) \\
&=  \sch_C(T,\phi) - \sch_C(S^2_{w-v'}) + \sch_C(S^2_{\dil(w-v')}) \\
&= \sch_C(T,\phi) - 0 + \chi_{\dil(w-v')} \\
&= \sch_C(T,\phi) + \chi_{\dil(w-v')}.
\end{aligned} \]
\end{proof}

\subsection{Proof of \Thm{th:main}}
\label{ss:main}
Proposition \ref{p:stable} implies that any two elements of $R_{g,v}$ with the same Schur invariant are equivalent in some $R_{g+M,w}$.  In other words, two orbits in $R_{g,v}/\MCG_*(\Sigma_{g,n})$ that have the same Schur invariant eventually merge into one orbit inside some $R_{g+M,w}/\MCG_*(\Sigma_{g+M,n+2N})$.  Because $G$ is finite, all of the sets of orbits are finite, and so Lemmas \ref{l:puncture} and \ref{l:handle} imply that all of the handle and puncture stabilization maps are bijections when $g$ and $v$ are large enough.  We conclude that only a finite amount of merging must occur before the Schur invariant maps
\[ \sch_C: R_{g+M,w}/\MCG_*(\Sigma_{g+M,n+2N}) \to H_2(BG_C) \]
are injections.  This proves

\begin{proposition}
For all $w \in \mathbb{Z}_{\ge 0}^{C\sslash G \times \{+1,-1\}}$ and $g$ large enough, the Schur invariant yields an injective map
\[ \sch_C: R_{g,w}/\MCG_*(\Sigma_{g,n}) \to H_2(BG_C). \]
\qed
\label{p:injection}
\end{proposition}

At first glance, this proposition is very close to the statement of Theorem \ref{th:main}(1); however, the proposition allows negatively trivialized branch points, while Theorem \ref{th:main} does not.  We now use dilation to address this.   By combining the proposition and Lemma \ref{l:dilation}, it is straightforward to prove the following

\begin{lemma}
For all $v \in \mathbb{Z}_{\ge 0}^{C\sslash G}$ large enough, we can write $v=w_{dil}$ where $w \in \mathbb{Z}_{\ge 0}^{C\sslash G \times \{+1,-1\}}$ and the following conditions hold:
\begin{enumerate}
\item $\sch_C: R_{g,w}/\MCG_*(\Sigma_{g,n}) \to H_2(BG_C)$ is injective,
\item $\dil: R_{g,w}/\MCG_*(\Sigma_{g,n}) \to R_{g,v}/\MCG_*(\Sigma_{g,n+N})$ is surjective,
\item $w$ is a stabilization of positive branching data $v' \in \mathbb{Z}_{\ge 0}^{C\sslash G \times \{+1\}}$, and
\item $v'$ is larger than $\sum_{\overline{c} \in C\sslash G} U^{\overline{c}}$. \qed
\end{enumerate} 
\label{l:final}
\end{lemma}

Now suppose $v \in \mathbb{Z}_{\ge 0}^{C\sslash G}$ is large enough for the lemma to hold.  Let $\phi,\psi \in R_{g,v}$ and suppose $\sch_C \phi = \sch_C \psi$.  By condition 2, there are $(T',\phi')$ and $(S',\psi')$ in $R_{g,w}$ such that $\phi = \dil(T',\phi')$ and $\psi = \dil(S',\psi')$.  Lemma \ref{l:dilhom} shows
\[ \begin{aligned}
\sch_C(T',\phi') &= \sch_C(\phi) - \chi_{w-v'} \\ &= \sch_C(\psi) - \chi_{w-v'} = \sch_C(S',\psi')
\end{aligned} \]
By condition 1, $(T',\phi)$ and $(S',\psi')$ are in the same $\MCG_{g,n}$.  Condition 2 now implies that $\phi$ and $\psi$ represent the same orbit in $R_{g,v}/\MCG_*(\Sigma_{g,n+N})$.  This proves part 1 of Theorem \ref{th:main}.

To prove part 2, observe that when $C$ generates $G$, the second part of Proposition \ref{p:stable} says that we never need to introduce handle stabilizations in the above argument.  That is, two orbits in $R_{g,v}/\MCG_*(\Sigma_{g,n})$ that have the same Schur invariant merge into one orbit inside some $R_{g,w}/\MCG_*(\Sigma_{g,n+2N})$.  Now continue with the same argument to finish the proof of part 2.

Finally, we show that $R_{g,v}/\MCG_*(\Sigma_{g,n})$ is a torsor for $M(G)_C$ in the cases where $g$ and $v$ are large enough, and where $v$ is large enough and $C$ generates $G$.  At the end of \Sec{ss:homology} we described an exact sequence
\[ 0 \to M(G)_C \to H_2(BG_C) \to N \to 0, \]
where the map $H_2(BG_C) \to N$ is the homological branch type.  All of the elements of $R_{g,v}$ have the same branch type, so, in particular, they have the same homological branch type $[v] \in N$.  Thus, for any $(T,\phi), (S,\psi) \in R_{g,v}$,
\[ \sch_C(T,\phi) - \sch_C(S,\psi) \in M(G)_C.\]
To conclude, we must show that every element of $M(G)_C$ can be represented by such a difference.

We first consider the case where $g$ and $v$ are large enough.  Fix a branched cover of the oriented disk with branch type $v$.  Since $N$ is the kernel of the evaluation map $\mathbb{Z}^{C \sslash G} \to G_{ab}$, the boundary monodromy of this cover is in the commutator subgroup $[G,G]$.  Let $g_1$ be the commutator length of $G$, \ie
\[ g_1 = \max_{x\in [G,G]} cl(x), \]
where
\[ cl(x) = \min \{ l \mid \exists a_1,b_1,\dots,a_l,b_l \in G, [a_1,b_1]\cdot [a_l,b_l] = x \}. \]
We can extend the chosen branched cover of the disk to a $C$-branched $G$-cover $(T,\phi)$ of a closed oriented surface with the same branch type $v$ and genus $g_1$; if necessary, we can pick $g_1$ even larger so that $(T,\phi_1) \in R_{g_1,v}$ is connected.  Let $g_2$ be large enough so that every element $\chi \in M(G)_C$ can be represented by an unbranched $G$-cover $\phi_\chi$ of a closed surface of genus $g_2$.  Let $g_3=g-g_1-g_2$ and let $\phi_3$ be the trivial $G$-cover over a surface of genus $g_3$.  Let $(T,\psi_\chi)$ be any element of $R_{g,v}$ equivalent to the connect sum
\[ (T,\phi) \# \phi_\chi \# \phi_3. \]
The Schur invariant of this cover is
\[\sch(T,\psi_\chi) = \sch_C(T_1,\phi_1) + \chi. \]
In particular, for every $\chi \in M(G)_C$,
\[ \sch(T,\psi_\chi) -\sch(T,\psi_0) = \chi. \]

Finally, suppose $v$ is large enough and $C$ generates $G$.  Lemma \ref{l:unbranched} shows that every element of $M(G)_C$ can be represented by an unbranched cover of $S^2$.  Better yet, there exists a branch type $v_1 \in \mathbb{Z}_{\ge 0}^{C\sslash G \times \{+1,-1\}}$ such that every element $\chi \in M(G)_C$ can be represented by an element $(T_\chi, \phi_\chi)$ of $R_{0,v_1}$.  Fix any element $(T_1,\phi_1)$ in $R_{0,v-v_1}$.  For every $\chi \in M(G)_C$, let $(T,\psi_\chi)$ be an element of $R_{0,v}$ that is equivalent to the connect sum
\[ (T_\chi, \phi_\chi) \# (T_1,\phi_1) \]
As before,
\[\sch(T,\psi_\chi) = \sch_C(T_1,\phi_1) + \chi, \]
and in particular,
\[ \sch(T,\psi_\chi) -\sch(T,\psi_0) = \chi \]
for every $\chi \in M(G)_C$.
\qed

\section{Outlook}
\label{s:outlook}
\subsection{Understanding the stable range}
\label{ss:unstable}
We remark that while Lemmas \ref{l:handle}, \ref{l:puncture} and \ref{l:dilation} provide specific lower bounds, the merging argument in the proof of \Thm{th:main} does not give any bounds whatsoever on when merging stops.  Thus, we do not have an upper bound for when the stable range begins.  The problem of computing the stable range is closely related to the \emph{quantitative Steenrod problem} in $H_2(G)$, \ie the question of finding a smallest genus representative of a given homology class in $H_2(G)$.  It could be interesting to try to bound the stable range for specific families of finite groups, such as solvable groups or the non-abelian simple groups.

\subsection{Applications to $G$-equivariant TQFT}
\label{ss:motivation}
The subject of this paper first received attention as early as Nielsen in 1937 \cite{Nielsen:periodisher}.  Despite its age, it takes on a new life in the context of topological quantum field theory.

The linearization of the permutation action of $\MCG_*(\Sigma_{g,n})$ on $\hR_{g,n}$ is closely related to one of the simplest examples of a topological quantum field theory, the untwisted Dijkgraaf-Witten theory with gauge group $G$ \cite{DijkgraafWitten:gauge, FreedQuinn:finite}.  In fact, this linearized action is precisely the representation of $\MCG_*(\Sigma_{g,n})$ afforded by the $G$-crossed modular tensor category of $G$-graded vector spaces.  Every $G$-crossed modular tensor category gives rise to an extended, pointed $(2+1)$-dimensional homotopy quantum field theory with target $K(G,1)$, and, hence, a (projective) representation of $\MCG_*(\Sigma_{g,n})$ on a vector space
\[ V^\times(\Sigma_{g,n}) = \bigoplus_{\phi \in \hR_{g,n}} V^\times(\Sigma_{g,n},\phi)\]
consisting of blocks indexed by the $G$-representation set $\hR_{g,n}$. \cite{Turaev:HQFT, TuraevVirelizier:HQFT2}.  The linear action of $\MCG_*(\Sigma_{g,n})$ on $V^\times(\Sigma_{g,n})$ refines the permutation action of $\MCG_*(\Sigma_{g,n})$ on the index set $\hR_{g,n}$.

In recent works, the author and Kuperberg study the action of $\MCG_*(\Sigma_{g,n})$ on $\hR_{g,n}$ under the assumption that $G$ is a nonabelian simple group \cite{KuperbergSamperton:zombies, KuperbergSamperton:coloring}.  In this case, we understand more than just the orbits of the action: building on \cite{DunfieldThurston:random} and \cite{RobertsVenkatesh:full}, the results of \cite{KuperbergSamperton:coloring} and \cite{KuperbergSamperton:zombies} establish a precise version of \emph{classical} topological computing via this action.  Complexity-theoretic hardness results for combinatorial 3-manifold invariants ensue.  As we now explain, we advertise this result here because we believe it may be of interest in the study of topological quantum computing with symmetry enriched topological phases (see \cite{DelaneyWang:defects} for an introduction).

Recall that the algebraic model of a $(2+1)$-dimensional $G$-symmetry enriched topological phase of matter ($G$-SET phase) is believed to be a $(2+1)$-dimensional $G$-crossed unitary modular tensor category \cite{BBCW:gauging}.  The physical interpretation says that a $G$-representation $\phi \in \hR_{g,n}$ is a background field that gauges some internal $G$-symmetry of a topological phase residing on $\Sigma_{g,n}$.  The subspaces $V^\times(\Sigma_{g,n},\phi)$ are called twisted sectors.  In the language of condensed matter physics, this paper seeks to understand when two gauge fields on $\Sigma_{g,n}$ are equivalent under a modular transformation.  This question recently received some attention in the physics literature, where the case of $n=0$ and $G$ abelian was solved \cite{BarkeshliWen:Z2, BBCW:gauging}.  (We note that the cyclic case was already known to Nielsen \cite{Nielsen:periodisher}, and the abelian and metacyclic cases were solved by Edmonds in the early 1980s \cite{Edmonds:symmetry1, Edmonds:symmetry2}.)

When every twisted sector is 1-dimensional, a $G$-SET phase is instead called a $G$-symmetry protected topological phase, or $G$-SPT phase.  It is well-known that topological quantum computing with a $G$-SPT phase is never quantum universal.  Nevertheless, the results of \cite{KuperbergSamperton:zombies, KuperbergSamperton:coloring}  imply that, at least for nonabelian simple $G$, a kind of topological computing with a $G$-SPT phase is $\shP$-complete via parsimonious reduction, a precise notion of classical (\emph{i.e.}~non-quantum) universality.  More generally, it follows that for such $G$, every $G$-SET phase can model classical reversible circuits.

Accordingly, when $G$ is nonabelian simple, one might hope that for the topological operations available from a $G$-SET phase to be quantum universal, the only thing left to find is a single entangling gate between two states in different twisted sectors.  Unfortunately, it is expected that no such entangling gate exists using topological operations.  Moreover, unless one is willing to believe that quantum computers can efficiently solve problems in $\shP$, the results of \cite{KuperbergSamperton:zombies, KuperbergSamperton:coloring} can be understood as evidence that preparing, measuring and topologically manipulating arbitrary gauge fields on a surface is too hard for a quantum computer to do efficiently.  Nevertheless, knowledge of the action of $\MCG_*(\Sigma_{g,n})$ on $\hR_{g,n}$ could help when designing protocols for universal gate sets augmented with \emph{non-topological} operations, as in \cite{DelaneyWang:defects}.

\providecommand{\bysame}{\leavevmode\hbox to3em{\hrulefill}\thinspace}
\providecommand{\href}[2]{#2}

\typeout{get arXiv to do 4 passes: Label(s) may have changed. Rerun}

\end{document}